\documentclass[11pt]{article}

\usepackage[margin=1in]{geometry}

\usepackage{amsfonts}
\usepackage{amsmath}
\usepackage{amsthm} % for proof environment
\usepackage{amssymb}
\usepackage{epsfig, graphics}
\usepackage{hyperref}
\usepackage{url}
\usepackage{color}
\usepackage{setspace}
\usepackage{float}
\usepackage{subfig}
\usepackage{colortbl}
\usepackage{bm}
\usepackage{bbm}

\usepackage[square, numbers, comma, sort&compress]{natbib}

\def\bd{{\mathbf d}}

\newcommand{\cN}{{\mathcal{N}}}

\def\cGn{\cG^{(n)}(\bd_n^+,\bd_n^-)}

\def\bS{\mathbf{S}}
\def\bD{\mathbf{D}}

\def\EE{\mathbb{E}}
\def\PP{\mathbb{P}}
\def\NN{\mathbb{N}}
\def\RR{\mathbb{R}}

\def\ind{{\rm 1\hspace{-0.90ex}1}}
\def\Var{\mathrm{Var}}

\newcommand{\Bin}{\mathsf{Bin}}

\def\top{\stackrel{p}{\longrightarrow}}

\newcommand{\cS}{\mathcal{S}}
\newcommand{\cI}{\mathcal{I}}
\newcommand{\cX}{\mathcal{X}}

\newcommand{\cD}{\mathcal{D}}

\newcommand{\cA}{\mathcal{A}}
\newcommand{\cR}{\mathcal{R}}

\newcommand{\cE}{\mathcal{E}}

\newcommand{\cH}{{\mathcal{H}}}

\newcommand{\cG}{{\mathcal{G}}}
\newcommand{\fR}{\mathfrak{R}}
\newcommand{\fF}{\mathfrak{F}}

\def\ind{{\rm 1\hspace{-0.90ex}1}}

\def\top{\stackrel{p}{\longrightarrow}}

\newtheorem{theorem}{Theorem}[section]
\newtheorem{remark}[theorem]{Remark}

\newtheorem{example}[theorem]{Example}
\newtheorem{assumption}{Assumption}
\newtheorem{definition}[theorem]{Definition}

\newtheorem{lemma}[theorem]{Lemma}
\newtheorem{claim}[theorem]{Claim}

\date{}

\begin{document}

%%%%%%%%%%%%%%%%

\title{Ruin Probabilities for Risk Processes in Stochastic Networks}

\author{
Hamed Amini \thanks{Department of Industrial and Systems Engineering, University of Florida, Gainesville, FL, USA, email:  aminil@ufl.edu} \and
 Zhongyuan Cao 
\thanks{INRIA Paris,  2 rue Simone Iff, CS 42112, 75589 Paris Cedex 12, France, and Universit\'e Paris-Dauphine, email: zhongyuan.cao@inria.fr }
\and 
Andreea Minca\thanks{Cornell University, School of Operations Research and Information Engineering, Ithaca, NY, 14850, USA, email: {\tt acm299@cornell.edu}}
\and
Agn\`es  Sulem
\thanks{INRIA,  Paris,  2 rue Simone Iff, CS 42112, 75589 Paris Cedex 12, France, email: agnes.sulem@inria.fr}}

\maketitle

\begin{abstract}
We study multidimensional Cram\'er-Lundberg risk processes where agents, located on a large sparse network, receive losses form their neighbors. To reduce the dimensionality of the problem, we introduce classification of agents according to an arbitrary countable set of types.   
The ruin of any agent triggers losses  for all of its neighbours. We consider the case when the loss arrival process induced by the ensemble of ruined agents follows a Poisson process with general intensity function that scales with the network size. When the size of the network goes to infinity,  we provide  explicit ruin probabilities at the end of the loss propagation process for agents of any  type. These limiting probabilities depend, in addition to the agents' types and the network structure, on the loss distribution and the loss arrival process. For a more complex risk processes on open networks, when in addition to the internal networked risk processes the agents receive losses from external users, we provide bounds on ruin probabilities.

\bigskip

\noindent {\bf Keywords:}  Risk processes, Cram\'er-Lundberg model,  ruin probabilities, stochastic networks.

\end{abstract}

%\newpage

\section{Introduction}

The study of ruin probabilities in stochastic networks is of paramount importance due to its real-world applications. For example, consider the case of financial risk management in the banking sector, where banks are connected through a network of interbank exposures. The ruin of one bank could trigger a chain reaction of losses throughout the network, leading to a systemic risk. Understanding the ruin probabilities of each bank, as well as the impact of different network structures and loss distributions, can help financial regulators to elaborate effective risk management strategies to mitigate the likelihood of a systemic financial crisis. Similarly, in the context of insurance, understanding ruin probabilities is crucial in the design of insurance products and in the assessment of solvency for insurance companies. The study of ruin probabilities can help them to set reserves and premiums that are commensurate with their risks and can also assist regulators in ensuring that insurance companies are adequately capitalized.

The classical compound risk process with Poisson claim arrivals, or the Cram\'er-Lundberg model (\cite{cramer1955collective, lundberg1903approximerad}) has been extensively used in quantitative risk management, see e.g., \cite{delbaen1987classical, mcneil2015quantitative}. In this model, the aggregate capital of an insurer who starts with initial capital $\gamma$, premium rate $\alpha$ and (loss) claim sizes $(L_k)$ is given by the following spectrally negative compound Poisson process
\begin{align}\label{eq:classicRP}
C(t) = \gamma + \alpha t -\sum_{k=1}^{\cN(t)} L_k,
\end{align}
where $L_k, k\in \NN$, are i.i.d. non-negative random variable following a distribution $F$ with mean $\bar{L}$ and $\cN(t)$ is a Poisson process with intensity $\beta>0$ independent of $L_k$. The ruin time for the insurer with initial capital  $\gamma$ is defined by 
$$\tau(\gamma) := \inf\{t\mid C(t)\le 0\},$$
(with the convention that $\inf\emptyset = \infty$) and the central question is to find the ruin probability
$$\psi(\gamma):=\PP(\tau(\gamma)<\infty).$$
 It is known (see e.g.~\cite{asmussen2010ruin, embrechts2013modelling}) that whenever $\beta \bar{L}>\alpha$, we have $\psi(\gamma)=1$ for all $\gamma\in \RR$ and whenever $\beta \bar{L} < \alpha$, the ruin probability can be computed using the famous Pollaczek–Khinchine formula as
\begin{align}
\psi(\gamma)=\left(1-\frac{\beta \bar{L}}{\alpha}\right)\sum_{k=0}^\infty \left(\frac{\beta \bar{L}}{\alpha}\right)^k\left(1-\widehat{F}^{*k}(\gamma)\right),
\label{eq:classickramer}
\end{align}
where 
\[\widehat{F}(\gamma)=\frac{1}{\bar L} \int_0^\gamma \bigl(1-F(u)\bigr) du,\]
and the operator $(\cdot)^{*k}$ denotes the $k$-fold convolution.

In this paper we consider a (stochastic) large sparse network setting that replaces the standalone jump process in the classical model.
In our model, losses of any firm do not occur via an exogenous Poisson process, but are due to the ruin of its neighbouring nodes (agents).

While the classical Cram\'er-Lundberg model and its generalizations such as Sparre  Andersen model (\cite{andersen1957collective}) have been  a pillar for collective risk theory for many decades,
ruin problems beyond one dimension remain challenging. Even numerical approximations are only available for some distributions,  in low dimensions (see e.g., \cite{avram2008exit, avram2008two}) and for some particular hierarchical structures.
For example, \cite{avram2017central} consider a central branch with one subsidiary which do not admit exact solution due to their complex dependent Sparre
Andersen structure. The authors propose  approximation techniques based on replacing the underlying structure with spectrally negative Markov additive processes. In \cite{behme2020ruin, kley2016risk} the authors consider loss propagation in bipartite graphs. For exponential claims, the authors provide Pollaczek–Khinchine formulas for the summative ruin probability of a group of agents. 
One should note that this is not a model of agent interrelation, but a model in which insurers connect to `objects' (external risks). As such, notions of losses that propagate from one ruined agent to another agent do not gave a correspondence in that model.

Network based models have been used  to advance  risk assessment in financial systems. 
An initial body of literature is concerned with economic questions such as  the effects of network structure on financial contagion~\cite{hurd2016contagion, amini2016inhomogeneous, acemoglu2015systemic, gai2010contagion, cont2010network} or with considering and integrating variations of the distress propagation mechanism, see e.g., \cite{glasserman2016contagion, veraart2020distress, benoit2017risks, elliott2014financial}. 
The question of control of financial contagion is posed in \cite{amini2015control, amini2017optimal}, where authors introduce a link revealing filtration and adaptive bailouts  mitigate the extent of contagion.
Many of the initial models are static, in the sense that there is one snapshot of the network and an initial wave of defaults leads to a second wave of defaults and so on. The state of the network does not change over time. Rather, it is reassessed in rounds, in order to find a final set of defaults.
Dynamic contagion is considered for example in \cite{capponi2020dynamic, feinstein2021dynamic}, where nodes are endowed with stochastic processes,  usually jump-diffusions.

This paper originates with the asymptotic analysis in \cite{amcomi-res}. Their results are purely static and their focus is to provide a condition of resilience under which the contagion does not spread to a strictly positive fraction of the agents. In
\cite{amicaosu2021limit, amicaosu2022timechange}, we provide steps for the asymptotic fraction of the ruined agents when nodes may  have special ruin dynamics without growth; 
Unlike   \cite{amcomi-res}, the analysis is not based on the differential equation method. Instead, we  generalize the
law of large numbers for a model of default cascades in the configuration model. 
This allows us to go beyond the static model, and the approach could be used to provide central limit theorems.
Our dynamic leads to a special case of loss intensity. All outgoing half-edges are assigned an exponential clock with parameter one, which determines when losses are revealed. The loss reveal intensity function is then equal to the number of remaining outgoing half-edges at time $t$.
This model can be seen as a variant of the notion of Parisian ruin: indeed, an agent could have become ruined if the incoming loss from a ruined neighbour was observed instantly. Instead, it is observed after an exponential time, during which the agents' capital increases. Therefore, it could well be that it withstands the loss at the time when it is revealed.

Closest to our paper is \cite{amini2022dynamic}. They allow to linear growth  at node level (proportional to the number of the node's links), while the loss reveal intensity is
assumed to be constant equal to the size of network. The key feature that allowed the asymptotic analysis was the fact that each link carried a constant loss.  
This in turn, made their analysis simpler.
Because losses can be arbitrary, the risk process does not evolve according to a given grid as in the case of constant losses. When a firm suffers a loss due to a neighbor failure  it will  move to a lower level of value. Here,  in presence of general losses, all possible value levels are coupled.

The remaining open question, posed in \cite{amini2022dynamic}, is to allow for losses that come from a general distribution, as opposed to constant losses.
In this paper, we solve this  open question, and in doing so, we bridge the risk literature on multi-dimensional risk processes with the financial network literature, and use the asymptotic results of the latter to provide the approximations of the ruin probabilities.

\paragraph{Outline.} 
In Section~\ref{sec:model} we introduce the model of interconnected risk processes. The model is driven by a classification of nodes according to types, whereas the interconnections occur according to the configuration model. In
Section~\ref{sec:main} we state our main results, concerning the asymptotic fraction of ruined nodes.
Section~\ref{sec:proofs} provides the proofs.
In Section~\ref{sec:extensions} we outline a complex risk process driven by both exogenous individual external loss processes and an internal risk processes where losses propagate in the network.
Section~\ref{sec:conc} concludes and proposes several open questions.

\paragraph{Notations.} Let $\{ X_n \}_{n \in \mathbb{N}}$ be a sequence of real-valued random variables on a probability space
$ (\Omega, \mathcal{F}, \mathbb{P})$. If $c \in \mathbb{R}$ is a constant, we write $X_n \stackrel{p}{\longrightarrow} c$ to denote that $X_n$ converges in probability to $c$. %That is, for any $\epsilon >0$, we have $\mathbb{P} (|X_n - c|>\epsilon) \longrightarrow 0$ as $n \longrightarrow \infty$. 
Let $\{ a_n \}_{n \in \mathbb{N}}$ be a sequence of real numbers that tends to infinity as $n \longrightarrow \infty$. We write $X_n = o_p (a_n)$, if $|X_n|/a_n \top 0$. 
If $\mathcal{E}_n$ is a measurable subset of $\Omega$, for any $n \in \mathbb{N}$, we say that the sequence
$\{ \mathcal{E}_n \}_{n \in \mathbb{N}}$ occurs with high probability (w.h.p.) if $\mathbb{P} (\mathcal{E}_n) = 1-o(1)$, as
$n\rightarrow \infty$.
%We say that an event holds w.h.p. (with high probability), if it holds with probability tending to 1 as $n\to \infty$. 
For any function $f$ defined on $\RR^+:=[0,\infty)$, $\|f\|_{L^1}$ denotes the $L^1$-norm of $f$ on $\RR^+$. For any event or set $A$, we denote by $A^c$ the complement of $A$. $\Bin (k,p)$ denotes a binomial  distribution  corresponding to the number of
successes of a sequence of $k$ independent Bernoulli trials each having probability of success  $p$. We denote by $\ind\{\cE\}$ the indicator of an event $\cE$; it is 1 if $\cE$ holds and 0 otherwise. For $a,b \in \RR$, we denote by $a\wedge b:=\min\{a,b\}$.

\section{The Model}\label{sec:model}

\subsection{Networked risk processes}\label{sec:NRP}
Consider a set  of $n$ agents (e.g., firms, insurers, re-insurers, business lines, \dots) denoted by $[n]:= \{1,2,\dots,n\}$. Agents hold contractual obligations with each other. In our networked risk processes model, the interaction of the agents' capital processes occurs through the network of obligations. Upon the ruin  of an agent, we consider that it will fail on their obligations and the neighbouring agents will suffer losses due to this non fulfillment. Let $\cG_n$ be a directed graph on $[n]$. For two agents $i,j \in [n]$, we write $i\to j$  when there is a directed edge (link) from $i$ to $j$ in $\cG_n$, modeling the fact that $i$ has a contractual obligation to $j$. 
Similar to the classical model \eqref{eq:classicRP},  each agent $i$ is endowed with an initial capital $\gamma_i > 0$, while  $\alpha_i(t)$ is a continuous non-decreasing function  describing the premium accumulation for agent $i$.

Here, we assume that the capital is affected by an initial exogenous proportional shock $\epsilon_i\in[0,1]$  and $\delta_i$ represents the total value of claims held by end-users on agent $i$ (deposits).
It is then possible for an agent to fail after the exogenous shock if its initial capital is lower than the end-user claims, i.e., $C_i(0) := \gamma_i(1-\epsilon_i) + \alpha_i(0)- \delta_i\leq 0$. 
The set of fundamentally ruined agents is thus $$\cD(0):=\{i\in[n]: C_i(0) \leq 0\}.$$
Ruined agents affect their neighbours through the network of obligations.
The ruin time for agent $i\in[n]$ is $\tau_i:= \inf\{t\mid C_i(t)\le 0\}$ where 
we  consider the following risk process for the capital of agent $i$ with network interactions $\cG_n$: 
\begin{align}\label{eq:riskP}
C_i(t):=\gamma_i(1-\epsilon_i)+\alpha_i (t)-\delta_i-\sum_{j\in [n]: j\to i} L_{ji}\ind\{\tau_j+T_{ji} \leq t\}.
\end{align}

Here we denote by $T_{ji}$ the delay between firm $j$'s ruin and the time when a neighbour $i$ processes its losses from the unfulfilled obligations of $j$ to $i$. At time $\tau_j + T_{ji}$, node $j$ processes a loss $L_{ij}$. 
Similar to the classical model, we assume that the incoming losses of each agent $i\in [n]$ are i.i.d. random variables following some positive distribution $F_i$ (potentially depending on the agent $i$ characteristics).
 In order to compare with the classical model, we will assume that $T_{ji}$ are i.i.d. exponentially distributed with some parameter $\beta>0$, i.e.,  $T_{ji}\sim \mathrm{Exp}(\beta)$ for all $i,j \in [n]$. 
 
Consequently, the set of ruined agents at time $t\geq$ will be given by $$\cD(t):=\{i\in[n]: C_i(t)\leq 0\}=\{i\in[n]: \tau_i \leq t\}.$$ We assume that agent $i$ becomes inactive upon ruin, so that $C_i(t)=C_i(t\wedge\tau_i)$.  

\subsection{Node classification}\label{sec:class}
We consider a classification of agents into a countable (finite or infinite)  set of types $\cX$. We denote by $x_i\in \cX$  the  type of agent $i\in [n]$. 

Let us denote by $\mu_x^{(n)}$  the fraction of agents in class $x\in\cX$ in the network $\cG_n$. In order to study the asymptotics, it is natural to assume there is a limiting distribution of types.

\begin{assumption}\label{cond-limit}
We assume that for some probability distribution functions $\mu$ over $\cX$ and independent of $n$, we have that $\mu_x^{(n)} \to \mu_x$, as $n\to \infty$, for all $x\in \cX$.
\end{assumption}

Since the type space is countable,  we can assume without loss of generality that all agents of same type $x$ (for all $x\in \cX$)  have the same number of outgoing links, denoted by $d_x^-$, and the same number of incoming links, denoted by $d_x^+$.   

To reduce the dimensionality of the networked risk processes problem, we further assume that all parameters are type dependent. Namely, we assume that $\gamma_i=\gamma_x, \alpha_i(.)=\alpha_x(.)$ and $\delta_i=\delta_x$ for all agents $i\in[n]$ with $x_i=x$.
Note that the type space can be made sufficiently large (but countable) to incorporate a wide variety of levels for these parameters.

In particular, shocks are assumed to be independent  random variables with distribution function (cdf) $F_x^{(\epsilon)}$ and density function $f_x^{(\epsilon)}$  depending on the type of each agent. We then set  
\begin{align}
q_{x,0}:=1-F_x^{(\epsilon)}\bigl(\frac{\gamma_x+\alpha_x(0)-\delta_x}{\gamma_x}\bigr),
\end{align} 
which  represents the (expected) fraction of initially ruined agents of type $x\in \cX$. 

The distribution of incoming losses for each agent are also assumed to be   type dependent random variables. For all agents of type $x\in \cX$, the loss distribution function (cdf) is denoted by  $F_x$ and the probability density function (pdf) is $f_x$. Thus we have $F_i = F_{x_i}$ for any agent $i\in [n]$.
\begin{remark}[From loss distribution to threshold distribution]\label{Ex:thresholds}
The model, as introduced, can be equivalent to a model of dynamic failure thresholds inferred from the loss distribution. These random thresholds measure how many ruined neighbours can an agent withstand before being ruined due to the incurred losses. 
For $x\in \cX$, let $\epsilon_x$ be a random variable with distribution $F^{(\epsilon)}_x$ and $\{L_{x}^{(k)}\}_{k=1}^{\infty}$ be a set of i.i.d. positive continuous random variables with common cumulative distribution function (cdf) $F_{x}$. The threshold distribution function at time $t$  is defined as $q_{x,0}(t)=q_{x,0}$,
$$q_{x,1}(t):=\PP\bigl(0<\gamma_x(1-\epsilon_x)+\alpha_x (t)-\delta_x\leq L_{x}^{(1)}\bigr),$$
and for all $\theta \geq 2$,
\begin{align*}
q_{x,\theta}(t)& :=\mathbb{P}\bigl(L_{x}^{(1)}+\cdots+L_{x}^{(\theta-1)}<\gamma_x(1-\epsilon_x)+\alpha_x (t)-\delta_x\leq L_{x}^{(1)}+\cdots+L_{x}^{(\theta)}\bigr)
\end{align*}
represents the probability at time $t$ that an agent of type $x$ is ruined after $\theta$ neighbouring ruins.
Since $\alpha_x (t)$ is a non-decreasing function of $t$,  this threshold function will be (stochastically) non-decreasing. Note that when $\alpha_x(t)=0$ over each class $x\in \cX$, the results of~\cite{amicaosu2021limit, amicaosu2022timechange} could be applied by using the above threshold distributions. Note that in these works, the threshold distributions are static, which makes the model simpler to study. Here the distributions change over time because there is time-dependent growth. The closest model is \cite{amini2022dynamic} which consider fixed losses and exponential inter-arrival times with fixed parameter. In this paper, in order to study the general setup, we do not use these threshold distributions - which are given here only for comparison. Instead, our analysis relies on a sequence of random threshold times representing the times where thresholds hit  subsequent levels.
\end{remark}

\subsection{Configuration model}\label{sec:CM}
Given a set of agents $[n]:=\{1, \dots, n\}$, the degree sequences $\bd^+_n=(d_1^+, \dots, d_n^+)$ and $\bd^-_n=(d_1^-, \dots, d_n^-)$ such that $\sum_{i\in [n]} d_i^+ = \sum_{i\in[n]} d_i^-$,  we associate to each agent $i\in [n]$ two sets: $\cH^+_i$ the set of incoming half-edges and $\cH^-_i$ the set of outgoing half-edges, with $|\cH_i^+|=d^+_i$ and $|\cH^-_i|=d^-_i$. Let $\mathbb{H}^+=\bigcup_{i=1}^{n} \cH^+_i$ and $\mathbb{H}^-=\bigcup_{i=1}^{n} \cH^-_i$. A {\it configuration} is a matching of $\mathbb{H}^+$ with $\mathbb{H}^-$. When an out-going half-edge of agent $i$ is matched with an in-coming half-edge of agent $j$, a directed edge from $i$ to $j$ appears in the graph. The configuration model is the random directed multigraph which is uniformly distributed across all configurations. The random graph constructed by the configuration model is denoted by $\cGn$. 
%Observe that the self-loops may occur, these become rare as $n\to \infty$ (under some regularity conditions provided in Section~\ref{sec:limit}). 
It is easy to show that conditional on the multigraph being a simple graph, we obtain a uniformly distributed random graph with these given degree sequences denoted by $\cG_*^{(n)}(\bd_n^+,\bd_n^-)$. In particular, any property that holds with high
probability on the configuration model also holds with high probability conditional on this random graph being simple (for $\cG_*^{(n)}(\bd_n^+,\bd_n^-)$) provided that
$\liminf_{n\longrightarrow\infty}\mathbb{P}(\cGn \  \mbox{simple})>0$, see e.g.~\cite{hofstad16}. 

We consider the risk processes in random network $\cGn$ satisfying the following regularity condition on the average degrees.

\begin{assumption}\label{cond-avg}
We assume that, as $n\to \infty$, the first moment of degrees converges and is finite:
$$\lambda^{(n)}:=\sum_{x\in \cX} d_x^+ \mu_x^{(n)} = \sum_{x\in \cX} d_x^- \mu_x^{(n)} \stackrel{(\text{as } n\to\infty)}{\longrightarrow} \lambda:= \sum_{x\in \cX} d_x^+ \mu_x \in (0,\infty).$$
\end{assumption}

 \subsection{The loss reveal process}\label{sec:LRP}
In order to study the risk processes in random network $\cGn$, we construct the configuration model simultaneously as we run the ruin propagation model. Starting from the set of initially ruined agents, at each step, we only look at one interaction between a ruined agent and its counterparty in the configuration model. Note that the set of ruined agents either stays the same or augments with each such interaction. We first introduce some notations. 

We denote by $\bD^{(n)}(t)$ the set of ruined agents at time $t$ and set $D^{(n)}(t):=|\bD^{(n)}(t)|$. Similarly, $\bS^{(n)}(t)=[n]/\bD^{(n)}(t)$ denotes the set of solvent agents at time $t$ and we set $S^{(n)}(t)=|\bS^{(n)}(t)|$. For  $x\in \cX$, we denote by $\bS^{(n)}_{x}(t)$ the set of all solvent agents in class $x$ at time $t$ and set $S^{(n)}_{x}(t)=|\bS^{(n)}_{x}(t)|$.  Moreover, for $x\in \cX$ and $\theta=0, \dots, d^+_x$, we denote by $\bS^{(n)}_{x,\theta}(t)$ the set of all solvent agents in $\bS^{(n)}_{x}(t)$ with exactly $\theta$ ruined incoming neighbours at time $t$. Set $S^{(n)}_{x,\theta}(t)=|\bS^{(n)}_{x,\theta}(t)|$. We define similarly the sets $\bD^{(n)}_x(t)$ and $\bD^{(n)}_{x,\theta}(t)$ with corresponding sizes $D^{(n)}_x(t)$ and $D^{(n)}_{x,\theta}(t)$. %Note that $D^{(n)}(t)$ is a non-decreasing function. 
We call all outgoing half-edges  that belong to a ruined  agent the {\it ruinous}  half-edges.

We consider the following loss reveal process, extending the risk processes of Section~\ref{sec:NRP}.  

In this process, losses coming from ruined agents are revealed one by one. At each loss reveal  we look at the interaction between an outgoing half-edge of a ruined agent (ruinous outgoing half-edge) with an incoming half-edge of its counterparty. By virtue of the configuration model,  this is chosen uniformly at random among all remaining incoming half-edges. When the ruinous outgoing half-edge is matched, the counterparty incurs a random  loss, drawn from a distribution depending on the characteristics class $x\in \cX$ of the counterparty. If this amount of loss is larger than the remaining capital, this agent will become ruined and all its outgoing half-edges become ruinous. Note that the loss reveal process stops when all ruinous outgoing half-edges have been matched. We use the notation $W_n(t)$ to denote the number of remaining (unrevealed) ruinous outgoing half-edges at time $t$. The contagion thus stops at the first time when $W_n(t) = 0$.
%Note that we  suppressed the dependence of these parameters on the size of the network $n$, as it is clear from the context.  

We consider a general loss reveal intensity process, denoted by $\cR_n(t)$, to describe the intensity of loss reveal at time $t$. Namely, if a loss is revealed at time $t_1\in \RR_+$, then we wait an exponential time with parameter $\cR_n(t_1)$ for the next loss reveal. Between two successive loss reveals, the intensity does not change and $\cR_n(t)$ is a random piecewise function. Note that $\cR_n(t)$ could depend on the state of risk processes at time $t$. In particular,  the networked risk processes of Section~\ref{sec:NRP} in configuration model would be equivalent to this loss reveal process by setting $\cR_n(t)=\beta W_n(t)$. Indeed, each counterparty of a ruinous half-edge will be revealed after an exponential time with parameter $\beta$. When there are $W_n(t)$ such unrevealed counterparties, the next reveal will be given by the minimum of these exponential times.

Our model allows for a general form of the loss reveal intensity $\cR_n(t)$, provided that the following condition holds. Let $\tau^\star_n$ denote the stopping time defined as the first  time when the above loss reveal process ends.  This is the first time such that $W_n(\tau^\star_n)=0$.

\begin{assumption}\label{cond-R}
We assume that the loss intensity function $\cR_n$ satisfies $\cR_n(t)=0$ for $t > \tau^\star_n$, and $\cR_n(t)=n\fR(t)+o_p(n)$ for $t\leq\tau^\star_n$ with $\fR(t)$ continuous, positive and $\|\fR\|_{L^1}:=\int_0^\infty \fR(s) ds <\infty$.
\end{assumption}

 By  Theorem~\ref{thm:centrality}, this assumption holds for the risk processes of Section~\ref{sec:NRP}.
 In the next section we state the limit theorems for the general loss reveal process satisfying Assumption \ref{cond-R} and then we apply them to the particular risk processes of Section~\ref{sec:NRP}.

\section{Main Theorems}\label{sec:main}

\subsection{Asymptotic analysis of the general loss reveal process}\label{sec:general}
We consider the loss reveal process of Section~\ref{sec:LRP} satisfying Assumption~\ref{cond-R} on the random graph $\cGn$ which satisfies Assumptions~\ref{cond-limit}-\ref{cond-avg}. Let us denote by $\cI_n(t)$ a Poisson process with intensity $\cR_n(t)$ at time $t$. This represents the total number of  ruinous outgoing half-edges revealed before time $t$. Then  $\cI_n(\infty)$ represents the total number of ruinous outgoing half-edges that will be revealed if the reveal process continues forever. Since the total number of reveals is bounded from above by the total number of links in the network, we need to stop the Poisson process at $\cI^\star_n=\cI_n(\infty) \wedge (n\lambda^{(n)})$, where for $a,b \in \RR$, we denote by $a\wedge b :=\min\{a,b\}$. 
We define
$$t_\fR(\lambda):=\inf\{t\geq 0: \int_0^t \fR(s)ds\geq \lambda\}.$$ 
By convention, if $\|\fR\|_{L^1}\leq \lambda$ we set $t_\fR(\lambda):=\infty$.

For $z\in [0,1]$, denote by $T_n(z)$ the time needed to reveal $\lceil nz\rceil$ ruinous outgoing half-edges.  The following lemma gives the asymptotic results on $\cI_n(t)$, $\cI^\star_n$ and $T_n(z)$.

\begin{lemma}\label{lem:fraction}
Under Assumptions~\ref{cond-limit}-\ref{cond-avg}, and for any given loss intensity function $\cR_n$ satisfying Assumption~\ref{cond-R}, we have as $n\to\infty$,
\begin{align}\label{eq:L_limit} 
\sup_{t\geq 0}\bigl|\frac{\cI_n(t)}{n} -\int_0^t \fR(s)ds\bigr|\top 0, 
\quad \text{and} \quad  \frac{\cI_n^{\star}}{n}\top \|\fR\|_{L^1}\wedge \lambda.
\end{align}
Further, for all $0\leq a <b < \|\fR\|_{L^1}\wedge \lambda$ and as $n\to \infty$,
\begin{align}
T_n(b)-T_n(a)\top \int_a^b \frac{1}{\fR(s)}ds.
\end{align}
\end{lemma}

The proof of lemma is provided in Section~\ref{proof:lem:fraction}. 

We further define some functions that will be used later. Let for $z\in[0,1]$:
\begin{align*}
b(d,z,\ell):=& \PP(\Bin(d,z)= \ell) = \binom{d}{\ell} z^\ell (1-z)^{d-\ell}, \\
\beta(d,z,\ell):=& \PP(\Bin(d,z)\geq \ell) = \sum_{r=\ell}^d \binom{d}{r} z^r (1-z)^{d-r},
\end{align*}
and $\Bin(d,z)$ denotes the binomial distribution with parameters $d$ and $z$.

For each $x\in\cX$, we denote by $\bm{L}_x:=(L_x^{(1)}, \dots, L_x^{(d_x^+)})$ the sequence of independent random losses with distribution $F_x$ and let $\bm{\ell}_x=(\ell_x^{(1)},\ell_x^{(2)},\ldots,\ell_{x}^{(d^+_x)})$ be a realization of $\bm{L}_x$. For a given $\bm{\ell}_x$ and a given initial shock $\epsilon_x$  define for all $\theta=0, 1, \dots, d^+_x$,
\begin{equation}\label{eq:deftau}
\tau_{x,\theta}(\epsilon_x, \bm{\ell}_x):=\inf\bigl\{ t\geq 0: \gamma_x(1-\epsilon_x)+\alpha_x(t)-\delta_x\geq \sum_{i=1}^\theta\ell_{x}^{(i)}\bigr\}.
\end{equation}
The function $\tau$ can be seen as the time threshold function of the loss $\bm{\ell}_x$. Indeed, for  initial shock $\epsilon$ and upcoming sequence of losses $\bm{\ell}_x$, the threshold function $\tau_{x,k}(\epsilon, \bm{\ell}_x)$ is the smallest time  needed for a firm of type $x$ to have enough capital for absorbing these incoming $k$ losses. Recall that a firm of type $x$ has a capital growth function $\alpha_x$ and external debt $\delta_x$. Note that $\tau_{x,0}(\epsilon_x, \bm{\ell}_x)>0$ denotes the event that agent of type $x$ initially becomes ruined under the shock $\epsilon_x$.

\begin{example}\label{eg:const_loss}
Consider the simple case where there is no recovery for agents, i.e. $\alpha_x=0$ for all $x\in \cX$, and the amount of loss for each agent in the same class $x\in \cX$ is constant $\ell_x$. Then by the definition of $\tau_{x,\theta}$, we have
$$\tau_{x,\theta} = \left\{ \begin{array}{rcl} 0 & \mbox{if} & \theta<\lceil \frac{\gamma_x(1-\epsilon_x)-\delta_x}{\ell_x} \rceil, \\ \infty & \mbox{if} & \theta\geq \lceil \frac{\gamma_x(1-\epsilon_x)-\delta_x}{\ell_x} \rceil.
\end{array}\right.$$
This would be equivalent to type-dependent threshold contagion model which extends the bootstrap percolation model. In bootstrap percolation model, the threshold is fixed for all nodes, i.e. $\lceil \gamma_x/\ell_x \rceil = \theta$ for all $x\in \cX$. We refer to \cite{amini10, Am-nn, balpit07} for results on bootstrap percolation in configuration model.
%For each type $x\in\cX$, we can regard $\delta_x$ as the defaut threshold, which decides the maximum number of losses that an agent can stand, see e.g. \cite{amicaosu2022timechange, amcomi-res} for more details. Since there is no recovery, when the total accumulated losses surpass the initial capital, the agent becomes ruined. 
\end{example}

For a given positive density function $\fR: \RR^+_0\to \RR^+$ with $\|\fR\|_{L^1}< \infty$, $x\in\cX$ and $\theta=0,1, \dots, d^+_x$,
we let the survival probability be (for all $t \geq 0$)
\begin{align}\label{eq:PR}
P^{\fR}_{x,\theta}(t,\epsilon_x,\bm{\ell}_x):=\PP(\tau_{x,0}(\epsilon_x, \bm{\ell}_x)=0 , U^{\fR,t}_{(1)}>\tau_{x,1}(\epsilon_x, \bm{\ell}_x),\ldots,U^{\fR,t}_{(\theta)}>\tau_{x,\theta}(\epsilon_x, \bm{\ell}_x)),
\end{align}
with the convention $P^{\fR}_{x,0}(t,\epsilon_x,\bm{\ell}_x):=\PP(\tau_{x,0}(\epsilon_x, \bm{\ell}_x)=0)$ for all $x\in\cX$, where $U^{\fR,t}_{(1)},U^{\fR,t}_{(2)},\ldots,U^{\fR,t}_{(\theta)}$ are the order statistics of $\theta$ i.i.d. random variables $\{U^{\fR,t}_i\}_{i=1, \dots, \theta}$ with distribution
\begin{align}\label{eq:U}
\PP(U_i^{\fR,t}\leq y)=\frac{\int_0^y \fR(s)ds}{\int_0^t \fR(s)ds},\quad\quad y\leq t.
\end{align}

Knowing that a loss arrives before time $t$, the probability that it arrives before time $y$ is given by \eqref{eq:U}.
Then for a fixed sequence of given losses and initial shock, Equation~\ref{eq:PR} represents the probability that the firm survives at time $t$, given that there are $\theta$ by that time (It must then be that  the ordered losses arrived  after the successive threshold times).

\begin{remark}
The joint probability density of the order statistics $U^{\fR,t}_{(1)},U^{\fR,t}_{(2)},\ldots,U^{\fR,t}_{(\theta)}$ is given by 
$$f_{U^{\fR,t}_{(1)},U^{\fR,t}_{(2)},\ldots,U^{\fR,t}_{(\theta)}}(u_1,u_2,\ldots,u_\theta)=\theta!(\int_0^t\fR(s)ds)^{-\theta}\prod_{i=1}^\theta \fR(u_i),$$
for all $u_1, \dots, u_\theta \in [0,t]$.
\end{remark}

By integrating the conditional survival probability in Equation \eqref{eq:PR} with respect to the probability density function of $\epsilon_x$ and $\bm{\ell}_x$, we obtain the survival probability at time $t$ for any agent of type $x$ with $\theta$ incoming losses absorbed by $t$, denoted by $\cS^{\fR}_{x,\theta}(t)$. This is defined by $\cS^{\fR}_{x,0}(t):=1-q_{x,0}$ and 
\begin{align*}
\cS^{\fR}_{x,\theta}(t):=\int P^{\fR}_{x,\theta}(t,\epsilon_x,\bm{\ell}_x)f_x^{(\epsilon)}(\epsilon_x) f_{x}(\ell_x^{(1)})f_{x}(\ell_x^{(2)})\cdots f_{x}(\ell_x^{(\theta)})d\epsilon_x d\ell_x^{(1)}d\ell_x^{(2)}\cdots d\ell_x^{(\theta)},
\end{align*}
for $\theta=1, \dots, d_x^+$.
It can be equally written as 
\begin{align}\label{eqdef-beta}
\cS^{\fR}_{x,\theta}(t)=\EE[P^{\fR}_{x,\theta}(t,\epsilon_x,\bm{L}_x)]=\PP(\tau_{x,0}(\epsilon_x,\bm{L}_x)=0, U^{\fR,t}_{(1)}>\tau_{x,1}(\epsilon_x, \bm{L}_x),\ldots,U^{\fR,t}_{(\theta)}>\tau_{x,\theta}(\epsilon_x, \bm{L}_x)),
\end{align}
where $\bm{L}_x:=(L_x^{(1)}, \dots, L_x^{(d_x^+)})$ is a sequence of independent random losses with distribution $F_x$ and $\epsilon_x$ is an independent random variable with distribution $F_x^{(\epsilon)}$.

For a given positive density function $\fR: \RR^+_0\to \RR^+$ with $\|\fR\|_{L^1}< \infty$, we define 
\begin{align}
\phi^{\fR}(t):=\frac{\int_0^{t\wedge t_{\fR}(\lambda)} \fR(s)ds}{\lambda},%{\|\fR\|_{L^1}\wedge \lambda},
\end{align}
so that the binomial probability $b(d^+_x,\phi^{\fR}(t),\theta)$ represents the probability an agent of type $x$ suffers $\theta$ losses prior to time $t$.

For $t\geq 0$ and given $\fR$, we define the following functions which will be shown to be the limiting fractions of surviving and defaulted agents, respectively:
\begin{align}
f^\fR_S(t):=& \sum_{x\in \cX} \mu_x \sum_{\theta=0}^{d_x^+}b(d^+_x,\phi^{\fR}(t),\theta) \cS^\fR_{x,\theta}(t), \ \ f^\fR_D(t)=1-f^\fR_S(t).
%f^\fR_{H^+}(t):=& \sum_{x\in \cX} \mu_x q_x \sum_{\theta=0}^{d_x^+}(d^+_x-\theta)\cS^\fR_{x,\theta}(t), \ \ f^\fR_{I^+}(t) = \lambda (1-\phi^\fR(t))-f^\fR_{H^+}(t),\\
\label{eq:fs}
\end{align}

We further define the following function which will be shown to be fraction of remaining (unrevealed) ruinous outgoing half-edges at time $t$ (characterizing the contagion stopping time):
\begin{align}
f^\fR_{W}(t):=& \lambda(1-\phi^\fR(t)) - \sum_{x\in \cX} \mu_x d_x^-\sum_{\theta=0}^{d_x^+}b(d^+_x,\phi^{\fR}(t),\theta) \cS^\fR_{x,\theta}(t).
%f_{H^-}:=& \sum_{x\in \cX} \mu_x d_x^-\sum_{\theta=1}^{d_x^+}q_x(\theta) \beta\bigl(d_x^+,z, d_x^+-\theta+1\bigr), \ \ f_W(z):=\lambda z - f_{H^-}.
\label{eq:fw}
\end{align}

Now we can state the following limit theorem regarding the fraction of solvent and defaulted agents at time $t$. Recall that $\tau^\star_n$ is the stopping time at which the ruin propagation stops, i.e., this is the first time such that $W_n(\tau^\star_n)=0$.

\begin{theorem}\label{thm:centrality}
Under Assumptions~\ref{cond-limit}-\ref{cond-avg}, and for any given loss intensity function $\cR_n$ satisfying Assumption~\ref{cond-R}, we have as $n\to\infty$,
\begin{align*}
\sup\limits_{t\leq \tau^\star_n}\bigl|\frac{S_{x,\theta}^{(n)}(t)}{n}-\mu_x b(d^+_x,\phi^{\fR}(t),\theta) \cS^{\fR}_{x,\theta}(t)\bigr|\top 0.
\end{align*}
Further, as $n\to \infty$,
\begin{align*}
\sup\limits_{t\leq \tau^\star_n}\bigl|\frac{S^{(n)}(t)}{n}-f^{\fR}_S(t)\bigr|\top 0, & \ \   \sup\limits_{t\leq \tau^\star_n}\bigl|\frac{D^{(n)}(t)}{n}-f^{\fR}_D(t)\bigr|\top 0,
%\\ \sup\limits_{t\leq \tau_n}\bigl|\frac{H^+_n(t)}{n}-f^{\fR}_{H^+}(t)\bigr|\top 0, & \ \   \sup\limits_{t\leq \tau_n}\bigl|\frac{I^+_n(t)}{n}-f^{\fR}_{I^+}(t)\bigr|\top 0,
\end{align*}
and the process $W_n$ satisfies
\begin{align*}
\sup\limits_{t\leq \tau^\star_n}\bigl|\frac{W_n(t)}{n}-f^{\fR}_W(t)\bigr|\top 0.
\end{align*}
\end{theorem}

The proof of theorem is provided in Section~\ref{proof:thm:centrality}.

\begin{remark}
    
The above theorem generalizes the results of \cite{amini2022dynamic} to the case of a general loss reveal intensity function $\cR_n(t)$ satisfying Assumption \ref{cond-R}. In contrast, in \cite{amini2022dynamic},  the reveal intensity is assumed to be constant equal to the size of network $n$ and the recovery rate is proportional to the agent's connectivity $d^+_x$, for each type $x\in\cX$ (the type is simply the degrees of each agent in \cite{amini2022dynamic}). Namely, $\cR_n(t)=n$ and $\alpha_x(t)=\alpha d^+_x/\lambda t$ for all $x\in\cX$. In this case, $\fR(t)= 1$ for $t\in [0,\lambda]$ and $\fR(t)= 0$ for $t> \lambda$. Thus $\phi^\fR(t)=t/\lambda$ and for $\theta=1, \dots, d^+_x$, $U^{\fR,t}_{(1)},U^{\fR,t}_{(2)},\ldots,U^{\fR,t}_{(\theta)}$ become the order statistics of $\theta$ i.i.d. uniform random variables on $[0,\lambda]$. Also, the time threshold $\tau_{x,k}$ does not have a dependence on the loss $\bm{\ell}_x$ as in  our definition \eqref{eq:deftau}. It simplifies to 
$$\tau_{x,\theta}:=\inf\{ t\geq 0: \Theta_x+ \alpha t d_x^+/\lambda \geq \theta\},$$
where $\Theta_x$ is the random initial default threshold $\PP(\Theta_x=k)=q_{x,k}$ as defined in Remark~\ref{Ex:thresholds}. In this case, we can regard the default threshold as the number of losses each agent could absorb and we could recover the results of~\cite{amini2022dynamic}.
\end{remark}

In the case when there is no growth in the network and $\alpha_x(t)=0$, it is more convenient to characterize the above limit functions through the threshold distribution functions $q_{x,\theta}$ and $\bar{q}_x$ which could be defined as (similar to Remark~\ref{Ex:thresholds})
\begin{align*}
q_{x,\theta}(t)& :=\mathbb{P}\bigl(L_{x}^{(1)}+\cdots+L_{x}^{(\theta-1)}<\gamma_x(1-\epsilon_x)+\alpha_x (t)-\delta_x\leq L_{x}^{(1)}+\cdots+L_{x}^{(\theta)}\bigr),
\end{align*}
for $\theta=1,\ldots,d^+_x$, and 
$$\bar{q}_x:=\PP\bigl(L_x^{(1)}+\dots+ L_x^{(d^+_x)} <\gamma_x(1-\epsilon_x)+\alpha_x (t)-d_x \bigr)=1-\sum_{\theta=0}^{d_x^+} q_{x,\theta}.$$

For $t\geq 0$, we define the functions (which are the simplified versions of  those in \eqref{eq:fs} and respectively \eqref{eq:fw})
\begin{align*}
\widehat{f}^\fR_S(t):=& \sum_{x\in \cX}\mu_x \bigl[\sum_{\theta=1}^{d_x^+}q_{x,\theta} \beta\bigl(d_x^+, 1-\phi^\fR(t), d_x^+-\theta+1\bigr)+\bar{q}_x\bigr], \ \ \widehat{f}^\fR_D(t)=1-\widehat{f}^\fR_S(t),\\
%\widehat{f}^\fR_{H^+}(t):=& \sum_{x\in \cX} \mu_x q_x\bigl[\sum_{\delta=1}^{d_x^+}\pi_x(\delta) \sum_{\ell=d_x^+-\delta+1}^{d_x^+} \ell b\bigl(d_x^+,1-\phi^\fR(t),\ell\bigr)+\bar{\pi}_x d^+_x\bigr] , \ \ \widehat{f}^\fR_{I^+}(t) = \lambda (1-\phi^\fR(t))-\widehat{f}^\fR_{H^+}(t),\\
\widehat{f}^\fR_{W}(t):=& \lambda(1-\phi^\fR(t)) - \sum_{x\in \cX} \mu_xd_x^-\bigl[\sum_{\theta=1}^{d_x^+}q_{x,\theta} \beta\bigl(d_x^+,1-\phi^\fR(t), d_x^+-\theta+1\bigr)+\bar{q}_x\bigr].
\end{align*}

As a corollary of Theorem~\ref{thm:centrality}, we have the following theorem.

\begin{theorem}\label{coro:centrality}
Suppose there is no growth in the network, i.e., $\alpha_x=0$ for all $x\in \cX$. Under Assumptions~\ref{cond-limit}-\ref{cond-avg}, and for any given loss intensity function $\cR_n$ satisfying Assumption~\ref{cond-R}, we have as $n\to\infty$, 
\begin{align*}
\sup\limits_{t\leq \tau_n}\bigl|\frac{S_{x,\theta}^{(n)}(t)}{n}-\mu_x  b\bigl(d_x^+, \phi^\fR(t), \theta\bigr)\bigl(\sum_{\delta=\theta+1}^{d^+_x}q_{x,\theta}+\bar{q}_x\bigr)\bigr|\top 0
\end{align*}
for all $x\in \cX$ and $\theta=1, \dots, d_x^+$. 
Further, as $n\to \infty$,
\begin{align*}
\sup\limits_{t\leq \tau_n}\bigl|\frac{S^{(n)}(t)}{n}-\widehat{f}^{\fR}_S(t)\bigr|\top 0, & \ \   \sup\limits_{t\leq \tau_n}\bigl|\frac{D^{(n)}(t)}{n}-\widehat{f}^{\fR}_D(t)\bigr|\top 0,
%\\ \sup\limits_{t\leq \tau_n}\bigl|\frac{H^+_n(t)}{n}-\widehat{f}^{\fR}_{H^+}(t)\bigr|\top 0, & \ \   \sup\limits_{t\leq \tau_n}\bigl|\frac{I^+_n(t)}{n}-\widehat{f}^{\fR}_{I^+}(t)\bigr|\top 0,
\end{align*}
and the process $W_n$ satisfies
\begin{align*}
\sup\limits_{t\leq \tau_n}\bigl|\frac{W_n(t)}{n}-\widehat{f}^{\fR}_W(t)\bigr|\top 0.
\end{align*}
\end{theorem}

The proof of above theorem is provided in Section~\ref{proof:coro:centrality}.

Theorem~\ref{coro:centrality} generalizes  the limit theorem of \cite{amcomi-res}
to the dynamic case. It also generalizes the law of large numbers of \cite{amicaosu2021limit}, where the authors consider default cascades in configuration model such that all outgoing half-edges are assigned with an exponential clock with parameter one, leading to a loss reveal intensity function  equal to the number of remaining outgoing half-edges at time $t\geq 0$ and satisfying $\cR_n(t)=n\lambda e^{-t}+o_p(n)$.

In order to determine the ruin probabilities, we need to study the stopping time $\tau^\star_n$ when there are no more ruinous outgoing links in the system. Since from Theorem~\ref{thm:centrality}, the fraction of remaining ruinous outgoing half-edges converges to $f^{\fR}_W(t)$, we define
\begin{align}
 t^\star_\fR:=\inf \bigl\{t \in[0,1]: f^{\fR}_W(t)=0\bigr\}.
 \end{align}

We say that $t^\star_\fR<\infty$ is a stable solution of $f^{\fR}_W(t)=0$ if  there exists a small $\epsilon>0$ such that $f^{\fR}_W(t)$ is negative on $[t^\star_\fR, t^\star_\fR+\epsilon)$.
 We have the following lemma.
\begin{lemma}\label{lem:tau}
Under Assumptions~\ref{cond-limit}-\ref{cond-avg}, and for any given loss intensity function $\cR_n$ satisfying Assumption~\ref{cond-R}, we have as $n\to\infty$:
\begin{itemize}
\item If $t^\star_\fR<\infty$ is a stable solution of $f^{\fR}_W(t)=0$, then $\tau^\star_n\top t^\star_\fR$.
\item If $t^\star_\fR=\infty$, then $\tau^\star_n\top \infty$.
\end{itemize}
\end{lemma}

The proof of lemma is provided in Section~\ref{proof:lem:tau}.

We are now ready to provide the limit theorem about the final ruin probabilities. As a corollary of Theorem~\ref{thm:centrality} and Lemma~\ref{lem:tau} the following holds.

\begin{theorem}\label{thm:final}
Under Assumptions~\ref{cond-limit}-\ref{cond-avg}, and for any given loss intensity function $\cR_n$ satisfying Assumption~\ref{cond-R}, we have as $n\to\infty$:
%Let $t^\star$ be the smallest point such $\fF(t)=0$, namely
% $$\inf\{t\geq 0:\fF(t)=0\}.$$
% Then we have $t^\star=\infty$,
\begin{itemize}
\item[(i)]  If $\int_0^{t^\star_\fR} \fR(s) ds=\lambda$, then asymptotically all agents are ruined by the end of the loss propagation process, i.e.
$$D^{(n)}(\tau^\star_n)=n-o_p(n).$$

\item[(ii)] If $t^\star_\fR<\infty$ is a stable solution of $f^{\fR}_W(t)=0$ and $\int_0^{t^\star_\fR} \fR(s) ds<\lambda$, then the ruin probability of an agent of type  $x\in\cX$ converges to
$$\frac{D^{(n)}_{x}(\tau^\star_n)}{n\mu_x^{(n)}} \top 1-\sum_{\theta=0}^{d_x^+} b(d^+_x,\phi^{\fR}(t),\theta)  \cS^{\fR}_{x,\theta}(t^\star_\fR),$$
 and the total number of ruined agents satisfies
 $$D^{(n)}(\tau^\star_n)=n\sum_{x\in\cX}\mu_x(1-\sum_{\theta=0}^{d^+_x}\cS^{\fR}_{x,\theta}(t^\star_\fR))+o_p(n).$$  

\item[(iii)] If $t^\star_\fR=\infty$ and $\|\fR\|_{L^1}<\lambda$,  then the ruin probability of an agent of type  $x\in\cX$ converges to
$$\frac{D^{(n)}_{x}(\tau^\star_n)}{n\mu_x^{(n)}} \top 1-\sum_{\theta=0}^{d_x^+} b\bigl(d_x^+,\|\fR\|_{L^1}/\lambda, \theta\bigr) \cS^{\fR}_{x,\theta}(\infty),$$
 and the total number of ruined agents satisfies
 $$D^{(n)}(\tau^\star_n)=n\sum_{x\in\cX}\mu_x(1-\sum_{\theta=0}^{d^+_x}b\bigl(d_x^+, \|\fR\|_{L^1}/\lambda, \theta\bigr) \cS^{\fR}_{x,\theta}(\infty))+o_p(n),$$  
where %(for all $x\in \cX, \theta\in \NN$) 
$\cS^{\fR}_{x,\theta}(\infty)$ denotes the limit of $\cS^{\fR}_{x,\theta}(t)$ as $t \to \infty$.

\end{itemize}

 \end{theorem}
 
 The proof of theorem is provided in Section~\ref{proof:thm:final}.

\subsection{Ruin probabilities for the networked risk processes}\label{sec:ruin}
We consider the networked risk processes of Section~\ref{sec:NRP} on the random graph $\cGn$. In this case the loss reveal intensity function is totally determined by the number of ruinous outgoing half-edges, namely $\cR_n(t)=\beta W_n(t)$. 

In the previous section we have assumed 
 that the loss intensity function $\cR_n$ has a limit function $\fR$ which satisfies Assumption \ref{cond-R}. 
 We will now show that for the networked risk processes of Section~\ref{sec:NRP}, Assumption \ref{cond-R} holds and  there exists a unique limit function $\fR^\star$. We will take advantage of Theorem~\ref{thm:final} to show that $\fR^\star$ can be characterized as a fixed point solution, representing the limit of remaining ruinous links. 
 
 To obtain this existence and uniqueness result for $\fR^\star$, we need to consider a second moment condition for the degrees of the random graph $\cGn$.

\begin{assumption}
\label{cond-moment}
We assume that, as $n\to \infty$, 
$\sum_{i\in[n]}(d^+_i+d^-_i)^2=O(n)$.
\end{assumption}

In particular, the above %second 
assumption implies (by uniform integrability) Assmption~\ref{cond-avg}  
 and $\lambda^{(n)} \to \lambda$ as $n\to \infty$. In this case, since $\liminf_{n\longrightarrow\infty}\mathbb{P}(\cGn \  \mbox{simple})>0$, our limit theorems could be transferred to the uniformly distributed random graph with these degree sequences $\cG_*^{(n)}(\bd_n^+,\bd_n^-)$, see e.g.,~\cite{hofstad16}. 

We have the following theorem.

\begin{theorem}\label{thm:main}
Let $\mathbb{L}_{\lambda}(\mathbb{R}^+)$ be the space of all continuous positive integrable functions $f$ with $\|f\|_1\leq \lambda$. Suppose that the loss reveal intensity satisfies $\cR_n(t)=\beta W_n(t)$ and the network sequence $\{\mathcal{G}_n\}_{n\in\NN}$ satisfies Assumptions~\ref{cond-limit} and ~\ref{cond-moment}. Then we have:
\begin{itemize}
\item[(i)] There exists a unique solution $\fR^\star$ in $\mathbb{L}_{\lambda}(\mathbb{R}^+)$ with an initial value $\fR^\star(0)=\beta \sum_{x\in\cX}\mu_x d^-_x(1-q_{x,0})$ to the fixed point equation
$\fR=\beta\Psi(\fR),$
where $\Psi:\mathbb{L}_{\lambda}(\mathbb{R}^+)\to \mathbb{L}_{\lambda}(\mathbb{R}^+)$ is the map
$$\Psi(\fR)(t)= \lambda(1- \phi^\fR(t)) - \sum_{x\in \cX} \mu_x d_x^-\sum_{\theta=0}^{d_x^+}b(d^+_x,\phi^{\fR}(t),\theta) \cS^{\fR}_{x,\theta}(t).$$
\item[(ii)] As $n\to\infty$, we have
 \begin{align*}
\sup\limits_{t\leq \tau^\star_n}\bigl|\frac{\beta W_n(t)}{n}-\fR^\star(t)\bigr|\top 0,
\end{align*}
and consequently, 
\begin{align*}
\sup\limits_{t\leq \tau^\star_n}\bigl|\frac{S^{(n)}(t)}{n}-f^{\fR^\star}_S(t)\bigr|\top 0 & \ \ \text{and} \ \ \  \sup\limits_{t\leq \tau^\star_n}\bigl|\frac{D^{(n)}(t)}{n}-f^{\fR^\star}_D(t)\bigr|\top 0.
%\sup\limits_{t\leq \tau^\star_n}\bigl|\frac{H^+_n(t)}{n}-f^{\fR^\star}_{H^+}(t)\bigr|\top 0, & \ \   \sup\limits_{t\leq \tau^\star_n}\bigl|\frac{I^+_n(t)}{n}-f^{\fR^\star}_{I^+}(t)\bigr|\top 0.
\end{align*}
\end{itemize}
\end{theorem}

The proof of the theorem is provided in Section~\ref{sec:proof_main}.

We also have the following lemma which guarantees that $\tau^\star_n\top \infty$ in this model.

\begin{lemma}\label{lem:tau_W}
Let all assumptions in Theorem~\ref{thm:main} hold and $\fR^\star(0)=\beta \sum_{x\in\cX}\mu_x d^-_x(1-q_{x,0})>0$. Then we have $\tau^\star_n> (1-\epsilon) \log n$ with high probability for any $\epsilon>0$.
\end{lemma}

As a direct consequence of Theorem~\ref{thm:final} and Proposition~\ref{lem:tau_W}, we have $\tau^\star_n\top \infty$ as $n\to\infty$, and thus the final state belongs to point $(ii)$ in Theorem~\ref{thm:main}. In this case, $\cR_n(t)=\beta W_n(t)=n\fR^\star+o_p(n)$. If $\|\fR^\star\|_{L^1}=\lambda$, then asymptotically all agents become ruined during the cascade. Otherwise, if $\|\fR^\star\|_{L^1}<\lambda$, for any type $x\in\cX$, the ruin probability of agent of type $x$ is 
 $$\frac{D^{(n)}_{x,\theta}(\tau^\star_n)}{n\mu_x^{(n)}}=1-\sum_{\theta=0}^{d^+_x}b(d^+_x,\|\fR^\star\|_{L^1}/\lambda,\theta) \cS^{\fR^\star}_{x,\theta}(\infty)+o_p(1),$$
 and the total number of ruined agents satisfies
 $$D^{(n)}(\tau^\star_n)=n\sum_{x\in\cX}\mu_x(1-\sum_{\theta=0}^{d^+_x}b(d^+_x,\|\fR^\star\|_{L^1}/\lambda,\theta)\cS^{\fR^\star}_{x,\theta}(\infty))+o_p(n).$$

\section{Proofs}\label{sec:proofs}

%\subsection{Preliminary lemmas and theorems}
Before proving our main theorems, we will introduce some preliminary definitions and lemmas.

\subsection{Some preliminary results}
We begin by asserting the following claim. The proof is straightforward and we omit it.

\begin{claim}\label{rq:GC}
 Let $\bm{U}^\fR_x=(U^\fR_{x,1}, \dots, U^\fR_{x,d_x^+})$ be a vector of $d^+_x$ i.i.d. random variables with common distribution (similiar to Equation~\ref{eq:U} replacing $t$ by $t_{\fR}(\lambda)$)
$$\PP(U^\fR_{x,i}\leq y)=\frac{\int_0^y \fR(s)ds}{\int_0^{t_{\fR}(\lambda)}\fR(s)ds}, \ y\leq t_{\fR}(\lambda), i=1, \dots, d_x^+,$$
for all $x\in\cX$. Then $U^{\fR,t}_{(1)},U^{\fR,t}_{(2)},\ldots,U^{\fR,t}_{(k)}$ have the same distribution as that of the first $k$ order statistics of $\bm{U}^\fR_x$ conditioned on $t\in [U^\fR_{x,(k)}, U^\fR_{x,(k+1)})$, where $U^\fR_{x,(1)}\leq U^\fR_{x,(2)} \leq \dots U^\fR_{x,(d_x^+)} $ are the order statistics of elements of $\bm{U}^\fR_x$. Thus the probability measure of $U^{\fR,t}_{(k)}$, for all $k=1, \dots, d^+_x$ and $t\geq 0$, can be generated by the vector $\bm{U}^\fR_x$.
\end{claim}

Let $\bm{L}_x$ be a vector of $d^+_x$ i.i.d. random losses to an agent of type $x$ and $\bm{L}_{x,k}$ be the subvector of first $k$ positions. From~\eqref{eqdef-beta}, $\cS^\fR_{x,\theta}(t)$ can be regarded as $\PP_x H_{x,\theta,t}$, where $H_{x,\theta,t}: [0,+\infty)\times(\RR^+)^{d^+_x}\times[0,+\infty)^{d^+_x}\mapsto \RR $ is a measurable function defined as
$$H_{x,\theta,t}(\epsilon_x,\bm{\ell}_x,\bm{u}_x):=\ind\{u_{(k)}\leq t<u_{(k+1)}\}\ind\{\tau_{x,0}(\epsilon_x,\bm{\ell}_x)=0,u_{(1)}>\tau_{x,1}(\epsilon_x,\bm{\ell}_x),\dots,u_{(\theta)}>\tau_{x,\theta}(\epsilon_x,\bm{\ell}_x)\},$$
where $u_{(1)},u_{(2)},\ldots,u_{(d^+_x)}$ is the order statistics of $\bm{u}_x$ and $\PP_x$ is the probability measure on $[0,+\infty)\times(\RR^+)^{d^+_x}\times[0,+\infty)^{d^+_x}$ generated by $(\epsilon_x,\bm{L}_x,\bm{U}^\fR_x)$.

Let us define the class of functions $\mathcal{H}_{x,\theta}$ as the collection of $H_{x,\theta,t}$ for all $t\in\RR$, i.e.,
$\mathcal{H}_{x,\theta}:=\{H_{x,\theta,t}, t\geq 0\}.$
We will show that for each $x\in\cX$ and $\theta\leq d^+_x$, the class $\mathcal{H}_{x,\theta}$ is a Glivenko-Cantelli class with respect to $\PP_x$ defined as follows.

\begin{definition}
A class of functions (or sets) $\mathcal{H}$ is called a strong Glivenko-Cantelli class with respect to a probability measure $\PP$ if 
$$\sup_{H\in\mathcal{H}}|\PP_nH-\PP H| \stackrel{a.s.}{\longrightarrow} 0, \quad \quad \text{as} \ \  n\to \infty,$$
where $\PP_n$ is the empirical measure of $\PP$.
\end{definition}

We also need to introduce the following definitions. 

For a collection of subsets of $\Omega$, denoted by $\mathcal{H}$, with $\Omega$ being a space (usually a sample space), the $n$-th shattering coefficient of $\mathcal{H}$ is defined by 
$$S_\mathcal{H}(n):=\max_{x_1,\ldots,x_n\in\Omega} \mbox{card}\{\{x_1,\ldots,x_n\}\cap\cA,\cA\in \mathcal{H}\},$$
where $\mathrm{card}\{.\}$ denotes the cardinality of the set. Then the Vapnik–Chervonenkis dimension (or VC dimension) of $\mathcal{H}$ is defined as
$${\rm VC}(\mathcal{H}):=\max \{n\geq 1: S_\mathcal{H}(n)=2^n\}.$$
We need the following Glivenko-Cantelli theorem for finite VC classes; see e.g.~\cite{Dudley1991GC} for more details.

\begin{theorem}[Glivenko-Cantelli theorem for finite VC classes]\label{thm:GC}
Let $\PP$ be a probability measure on a Polish space $\Omega$ and $\mathcal{H}$ be a collection of subsets of $\Omega$ with ${\rm VC}(\mathcal{H})<\infty$. Then $\mathcal{H}$ is a strong Glivenko-Cantelli class with respect to $\PP$.
%A class of sets $\mathcal{H}$ is uniformly Glivenko-Cantelli if and only if it is a Vapnik–Chervonenkis class (with finite VC dimension).
\end{theorem}

As a result of above theorem, we have the following lemma.

\begin{lemma}\label{lem:GC}
For each $x\in\cX$ and $\theta=0,1, \dots, d^+_x$, $\mathcal{H}_{x,\theta}$ is a strong Glivenko-Cantelli class with respect to $\PP_x$.
\end{lemma}

\begin{proof}
For each $x \in \mathcal{X}$ and $\theta = 0, 1, \dots, d^+_x$, it can be deduced from Claim~\ref{rq:GC} that $\mathbb{P}_x$ is a probability measure generated by $(\epsilon_x,\bm{L}_x,\bm{U}^\fR_x)$, and is defined on a Polish space. It can be verified that ${\rm VC}(\mathcal{H}{x,\theta}) = 2$, making it a finite VC class. The lemma follows then from Theorem~\ref{thm:GC}. 
\end{proof}

We will also use the following well known result, see e.g.~\cite{ross1996stochastic} for a proof,  which offers us a way to convert the loss reveal process into a death process with an  initial number $\cI^\star_n$ and i.i.d. lifetimes. 
\begin{lemma}\label{lem:density}
Let $\{N(t): t\geq 0\}$ be an inhomogeneous Poisson process with an intensity function $\fR(t)$ and arrival times $\{\sigma_k, k\geq 1\}$. Given any fixed $t >0$ and conditional on $m$ arrivals before time $t$, the random vector $(\sigma_1,\sigma_2,\ldots,\sigma_m)$  has the same distribution as  the random vector $(Y_{(1)},Y_{(2)},\ldots,Y_{(m)})$, where $Y_{(i)}$ is the $i$-th order statistic of $m$ i.i.d. random variables with probability density function
$\fR(s)(\int_0^t \fR(u)du\bigr)^{-1}$.
\end{lemma}

We are now ready to present the proofs of our main theorems.

\subsection{Proof of Lemma~\ref{lem:fraction}}\label{proof:lem:fraction}

We first prove that for any $t\geq 0$, as $n\to \infty$,
$$\frac{\cI^{(n)}(t)}{n}\top \lambda\phi^{\fR}(t).$$
%First let prove that for the total number of revealed ruinous half-edges before time $t$, denoted by $Y^{(n)}(t)$, we have the following asymptotic relation 
%$$\frac{Y^{(n)}(t)}{n}\top \phi^{\fR}(t),$$
%for any $t\geq 0$.

Let $\Delta s$ be a small time interval and let $m=\lceil t/\Delta s\rceil$ represent the number of time intervals between $[0,t]$. Notice that, as $n\to \infty$,
\begin{align*}
\frac{\EE[\cI^{(n)}(t)]}{n} & =\frac{1}{n} \sum_{i=1}^{m}\EE[\cI^{(n)}(i\Delta s)-\cI^{(n)}((i-1)\Delta s)]
= \sum_{i=1}^{m} \Delta s \fR(i\Delta s)+o(1). 
\end{align*}
Hence for $t\leq t_{\fR}(\lambda)$ as $\Delta s\to 0$, we have
$$\frac{\EE[\cI^{(n)}(t)]}{n} \to \int_0^t \fR(s) ds.$$
On the other hand, by the property of Poisson distribution, we have
$$\EE[\cI^{(n)}(i\Delta s)-\cI^{(n)}((i-1)\Delta s)]=\Var [\cI^{(n)}(i\Delta s)-\cI^{(n)}((i-1)\Delta s)].$$
It therefore follows that
\begin{align*}
\Var(\frac{\cI^{(n)}(t)}{n}) & =\frac{1}{n^2} \sum_{i=1}^{m}\Var[\cI^{(n)}(i\Delta s)-\cI^{(n)}((i-1)\Delta s)] \\
& = \frac{1}{n}\bigl(\sum_{i=1}^{m} \Delta s \fR(i\Delta s)+o(1)\bigr) = o(1). 
\end{align*}
Thus by Chebysev's inequality we have, as $n\to\infty$, for any $t\geq 0$, 
\begin{equation}\label{eq:Y}
\frac{\cI^{(n)}(t)}{n}\top \lambda \phi^{\fR}(t).
\end{equation}

As a consequence, by letting $t\to \infty$, the final fraction of revealed ruinous outgoing half-edges converges to $\|\fR\|_{L^1}$ in probability, i.e.,
\begin{align}\label{eq:final_reveal}
\frac{\mathcal{I}^\star_n}{n}\top \lambda.
\end{align} 
Let $T(x)$ be defined as the inverse of $A_t$, with
$A_t:=\int_0^{t}\fR(s)ds,$
 and $A_{T(x)}=x$. Suppose that for some $\delta>0$ and $n$ large, $T_n(x)\geq T(x)+\delta$ or $T_n(x)\leq T(x)-\delta$. Then by \eqref{eq:Y}, one can show that, for some small $\epsilon>0$ and $n$ large enough, with high probability
$$x\leq \frac{\cI^{(n)}(T_n(x))}{n}\top \lambda \phi^{\fR}(T_n(x))\geq x+\epsilon,$$
or (respectively)
 $$x-\frac{1}{n}\geq\frac{\cI^{(n)}(T_n(x))}{n}\top \lambda \phi^{\fR}(T_n(x))\leq x-\epsilon.$$
By contradiction, we conclude that with high probability (for large $n$)
$$T(x)-\delta< T_n(x)<T(x)+\delta.$$
It therefore follows that, by taking $\delta$ arbitrarily small,
$T_n(x)\top T(x)$ as $n\to \infty$.

Since $A_t$ is continuously differentiable, by the inverse function theorem we have
$T'(x)=\frac{1}{\fR(x)}.$
Therefore, $T(x)=\int_0^x 1/\fR(s)ds$, and we can conclude that
$$T_n(b)-T_n(a)\top \int_a^b \frac{1}{\fR(s)}ds,$$
for all $0\leq a<b <\|\fR\|_{L^1}\wedge\lambda$. 

Now, we proceed to prove \eqref{eq:L_limit}. Note that the limit of the total fraction of revealed ruinous half-edges in the end is equal to $\|\fR\|_{L^1}\wedge \lambda$. 
%{\red It suffices to prove for $\cR_n(t)$ with $\|\fR\|_{L^1}=\lambda$.}
Then, using Lemma~\ref{lem:density}, we can transform the reveal process into a death process with an initial number of balls $\mathcal{I}^\star_n$ and i.i.d. lifetimes. Conditioned on $\mathcal{I}^\star_n$, the moments of revealing are simply the order statistics of  $\mathcal{I}^\star_n$ i.i.d. random variables with density function
\begin{equation}\label{eq:den}
f^{(n)}(t):=\frac{\cR_n(t)}{\|\cR_n\|_{L^1}\wedge(n\lambda^{(n)})}=\frac{\fR(t)}{\|\fR \|_{L^1}\wedge\lambda}+o_p(1),
\end{equation}
for ${t \leq t_{\fR}(\lambda)} $. By using dominated convergence theorem, the above equation implies that
\begin{equation}\label{eq:f_n}
 \sup_{t\geq 0}\bigl|\int_0^{t\wedge t_{\fR}(\lambda)} f^{(n)}(s)ds -\frac{\lambda}{\|\fR \|_{L^1}\wedge\lambda}\phi^\fR(t)\bigr|\top 0.
\end{equation}
 
By the above analysis, it is clear that $Y^{(n)}(t)$ is a pure death process with an initial number of balls $\cI^\star_n$ and  i.i.d. lifetimes with density $f^{(n)}(t)$, defined by \eqref{eq:den}. Thus, it follows from the Glivenko-Cantelli Theorem that
 \begin{equation}\label{eq:glivenko}
 \sup_{t\geq 0} \bigl| \frac{Y^{(n)}(t)}{\cI^\star_n}- \int_0^{t\wedge t_{\fR}(\lambda)}  f^{(n)}(s)ds\bigr|\top 0.
\end{equation}

Combining \eqref{eq:final_reveal}, \eqref{eq:f_n} and \eqref{eq:glivenko}, we obtain that
\begin{equation}\label{L_limit}
\begin{split}
\sup_{t\geq 0} & \frac{1}{\|\fR \|_{L^1}\wedge\lambda}\bigl|\frac{Y^{(n)}(t)}{n} - \int_0^{t\wedge t_\fR(\lambda)} \fR(s)ds\bigr| \\
& \leq \sup_{t\geq 0} \bigl| \frac{Y^{(n)}(t)}{n(\|\fR \|_{L^1}\wedge\lambda)}- \frac{Y^{(n)}(t)}{\cI^\star_n} \bigr|+ \sup_{t\geq 0} \bigl| \frac{Y^{(n)}(t)}{\cI^\star_n}- \int_0^{t\wedge t_{\fR}(\lambda)}  f^{(n)}(s)ds\bigr|\\
& \quad\quad + \sup_{t\geq 0}\bigl|\int_0^{t\wedge t_{\fR}(\lambda)}  f^{(n)}(s)ds -\phi^\fR(t)\bigr| \top 0,
\end{split}
\end{equation}
which implies that
\begin{align} 
\sup_{t\geq 0}\bigl|\frac{Y^{(n)}(t)}{n} - \int_0^{t\wedge t_\fR(\lambda)} \fR(s)ds\bigr|,
\end{align}
as desired.

\subsection{Proof of Theorem~\ref{thm:centrality}}\label{proof:thm:centrality}
Recall that $\cI_n(t)$ represents the number of revealed ruined outgoing half-edges at time $t$ before the stopping time $\tau^\star_n$. According to the construction of the configuration model, all pairs of outgoing and incoming half-edges are chosen uniformly at random. Even if the contagion stops before all half-edges are revealed, we continue the reveal process until the end for the sake of our analysis.
This will not change the process before $\tau^\star_n$ and hence will not affect our results.

Based on Lemma~\ref{lem:density}, conditioned on the total number of final revealed ruinous outgoing half-edges, $\cI^\star_n$, we can consider the contagion process as follows: There are $\cI^\star_n$ ruinous half-edges in total, each incoming half-edge pairs with a ruinous outgoing half-edge with probability $\cI^\star_n/n\lambda^{(n)}$ independently (see Remark~\ref{rq:little_change} below). If it pairs, this occurs after a random time from the start, with density $f^{(n)}(t)$, defined in \eqref{eq:den}.

 \begin{remark}\label{rq:little_change}
Note that the following two events are not generally equivalent:
\begin{itemize}
\item[(i)] Each incoming half-edge is paired with a ruinous outgoing half-edge with probability $\cI^\star_n/n\lambda^{(n)}$ independently;
\item[(ii)] All pairs of incoming and outgoing half-edges are matched uniformly at random, with a total of $\cI^\star_n$ ruinous outgoing half-edges.
\end{itemize}
However, as $n$ grows, the number of ruinous half-edges in $(i)$ will approach $\cI^\star_n$ with high probability, due to the strong law of large numbers. The total number of ruinous half-edges in $(i)$ would be $\cI^\star_n+o_p(n)$. This slight deviation does not affect the limit results.
\end{remark}

Let us define 
$$t^{(n)}_{\fR}(\lambda):=\inf\{t\geq 0: \frac 1 n \int_0^t \cR_n(s)ds\geq \lambda\}.$$ 

Denote $\PP^n_x$ as the probability measure generated by the vectors $\bm{L}_x$ and $\bm{U}^{\fR,(n)}_x$, where $\bm{U}^{\fR,(n)}_x$ is defined similarly as $\bm{U}^\fR_x$ in Claim~\ref{rq:GC} with the distribution
$$\PP(U^{\fR,(n)}_{x,i}\leq y)=\frac{\int_0^y \cR_n(s)ds}{\int_0^{t^{(n)}_{\fR}(\lambda)} \cR_n(s)ds}, \ \ \text{for} \ \ y\leq t^{(n)}_{\fR}(\lambda).$$

Note that $\PP^n_x$ is a random measure since it depends on $\cR_n$. Let $N^{(n)}_x$ denote the number of type $x$ agents in $\cGn$. By the analysis at the beginning of the proof, we know that, uniformly for all incoming half-edges, the probability of pairing with a ruinous half-edge before time $t$ is
$$\phi_n^\fR(t):=\frac{\cI^\star_n\int_0^{t\wedge t^{(n)}_{\fR}(\lambda)} \cR_n(s)ds}{n\lambda^{(n)}\int_0^{t^{(n)}_{\fR}(\lambda)}\cR_n(s)ds},$$
and the probability of never pairing with a ruinous half-edge is 
$\frac{n\lambda^{(n)}-\cI^\star_n}{n\lambda^{(n)}}.$

Under the probability measure $\PP^n_x$, the probability of there being $\theta$ loss arrivals before time $t$ is given by $b(d^+_x,\phi_n^\fR(t),\theta)$. Given $\theta$ loss arrivals before time $t$, the probability of a type $x$ agent being solvent at time $t$ is 
$$\PP^n_x\bigl(\tau_{x,0}(\epsilon_x,\bm{L}_x)=0,U^{\fR,t,(n)}_{(1)}>\tau_{x,1}(\epsilon_x,\bm{L}_x),U^{\fR,t,(n)}_{(2)}>\tau_{x,2}(\epsilon_x,\bm{L}_x),\ldots,U^{\fR,t,(n)}_{(\theta)}>\tau_{x,\theta}(\epsilon_x,\bm{L}_x)\bigr),$$
where $U^{\fR,t,(n)}_{(1)},U^{\fR,t,(n)}_{(2)},\ldots,U^{\fR,t,(n)}_{(\theta)}$ are the order statistics of $\theta$ i.i.d. random variable $\bigl(U^{\fR,t,(n)}_i\bigr)_{i=1}^\theta$ with distribution
$$\PP(U^{\fR,t,(n)}\leq y)=\frac{\int_0^y \cR_n(s)ds}{\int_0^t \cR_n(s)ds},\quad\quad y\leq t.$$
We define $\cS^{\fR,(n)}_{x,\theta}(t)$ as following:
$$\cS^{\fR,(n)}_{x,\theta}(t):=\PP^n_x(\tau_{x,0}(\epsilon_x,\bm{L}_x)=0,U^{\fR,t,(n)}_{(1)}>\tau_{x,1}(\epsilon_x,\bm{L}_x),\ldots,U^{\fR,t,(n)}_{(\theta)}>\tau_{x,\theta}(\epsilon_x,\bm{L}_x)).$$

By Claim~\ref{rq:GC}, for any $t\geq 0$, $x\in\cX$ and $0\leq\theta\leq d^+_x$, the fraction of solvent agents with exactly $\theta$ losses absorbed before time $t$, namely $S_{x,\theta}^{(n)}(t)/N^{(n)}_x$ is the mapping $\widetilde{\PP}^n_x H_{x,\theta,t}$ with respect to the empirical measure of $\PP^n_x$, denoted by $\widetilde{\PP}^n_x$. By Lemma~\ref{lem:GC}, we have for any stopping time $\tau_n\leq \tau^\star_n$,  
\begin{align}\label{eq:11}
\sup\limits_{t\leq \tau_n}\bigl|\frac{S_{x,\theta}^{(n)}(t)}{N^{(n)}_x}- b(d^+_x,\phi_n^{\fR}(t),\theta)\cS^{\fR,(n)}_{x,\theta}(t)\bigr|\top 0.
\end{align}

Combine the above analysis, clearly we have $\PP^n_xH_{x,\theta,t}=b(d^+_x,\phi_n^{\fR}(t),\theta) \cS^{\fR,(n)}_{x,\theta}(t)$. Notice that, as $n\to \infty$, we have for any $t\geq 0$,
$$\phi_n^\fR(t)=\frac{\cI^\star_n\int_0^{t\wedge t^{(n)}_{\fR}(\lambda)} \cR_n(s)ds}{n\lambda^{(n)}\int_0^{t^{(n)}_{\fR}(\lambda)}\cR_n(s)ds}\top\frac{\int_0^{t\wedge t_{\fR}(\lambda)} \fR(s)ds}{\lambda}=\phi^\fR(t) ,$$
and
$$\frac{n\lambda^{(n)}-\cI^\star_n}{n\lambda^{(n)}}\top \frac{\lambda-(\|\fR\|_{L^1}\wedge\lambda)}{\lambda}.$$
Combining with dominated convergence theorem, the above two equations give that
\begin{align}\label{eq:22}
\sup_{t\geq 0} |\PP^n_x H_{x,\theta,t}-\PP_x H_{x,\theta,t}|\top 0.
\end{align}
Hence by \eqref{eq:11} and \eqref{eq:22}, we obtain
$$\sup\limits_{t\leq \tau_n}\bigl|\frac{S_{x,\theta}^{(n)}(t)}{N^{(n)}_x}- b(d^+_x,\phi^{\fR}(t),\theta) \cS^{\fR}_{x,\theta}(t)\bigr|\top 0.$$ 

Note that by Assumption~\ref{cond-limit}, $N^{(n)}_{x}/n\longrightarrow\mu_x$ as $n\longrightarrow\infty$, for all $x\in\cX$, so we have 
\begin{align*}
\sup\limits_{t\leq \tau_n}|\frac{N^{(n)}_x}{n}\cS^\fR_{x,\theta}(t)- \mu_x \cS^{\fR}_{x,\theta}(t)\bigr|\to 0.
\end{align*}
Combine the two equations above, we obtain our first assertion
\begin{align}\label{S_limit}
\sup\limits_{t\leq \tau_n}\bigl|\frac{S_{x,\theta}^{(n)}(t)}{n}-\mu_x b(d^+_x,\phi^{\fR}(t),\theta)  \cS^\fR_{x,\theta}(t)\bigr|\top 0.
\end{align}

Let $\cX_K$ be the set of all characteristic $x\in \cX$ such that $d^+_x+d^-_x\leq K$. Since (by Assumption~\ref{cond-avg}) $\lambda\in(0,\infty)$, for arbitrary small $\varepsilon>0$, there exists $K_{\varepsilon}$ such that $\sum_{x\in\cX\setminus\cX_{K_{\varepsilon}}}\mu_x(d^+_x+d^-_x)<\varepsilon$. Further, by Assumption~\ref{cond-avg} and dominated convergence, $$\sum_{x\in\cX\setminus\cX_{K_{\varepsilon}}}(d^+_x+d^-_x)N^{(n)}_x/n\longrightarrow
\sum_{x\in\cX\setminus\cX_{K_{\varepsilon}}}(d^+_x+d^-_x)\mu_x<\varepsilon.$$
Hence for $n$ large enough, we have $\sum_{x\in\cX\setminus\cX_{K_{\varepsilon}}}(d^+_x+d^-_x)N^{(n)}_x/n<2\varepsilon$. By \eqref{S_limit}, we obtain
\begin{align*}
&\sup\limits_{t\leq \tau_n} \sum_{x\in\cX}(d^+_x+d^-_x)\sum_{\theta=0}^{d^+_x}\bigl|\frac{S_{x,\theta}^{(n)}(t)}{n}-\mu_x b(d^+_x,\phi^{\fR}(t),\theta)  \cS^\fR_{x,\theta}(t)\bigr| \nonumber\\
 \leq& \sup\limits_{t\leq \tau_n}\sum_{x\in\cX_{K_{\varepsilon}}}(d^+_x+d^-_x)\sum_{\theta=0}^{d^+_x} \bigl|\frac{S_{x,\theta}^{(n)}(t)}{n}-\mu_x b(d^+_x,\phi^{\fR}(t),\theta)  \cS^\fR_{x,\theta}(t)\bigr|\nonumber\\
  &+ \sup\limits_{t\leq \tau_n}\sum_{x\in\cX\setminus\cX_{K_{\varepsilon}}} (d^+_x+d^-_x)\sum_{\theta=0}^{d^+_x} \bigl|\frac{S_{x,\theta}^{(n)}(t)}{n}-\mu_x b(d^+_x,\phi^{\fR}(t),\theta)  \cS^\fR_{x,\theta}(t)\bigr|\nonumber\\
  \leq& o_p(1)+\sum_{x\in\cX\setminus\cX_{K_{\varepsilon}}}(d^+_x+d^-_x)(N^{(n)}_x/n+\mu_x)\leq o_p(1)+3\varepsilon.
\end{align*}
By taking $\varepsilon$ arbitrarily small, it follows that
\begin{align}\label{S_sum_limit}
\sup\limits_{t\leq \tau_n} \sum_{x\in\cX}(d^+_x+d^-_x)\bigl|\frac{S_{x,\theta}^{(n)}(t)}{n}-\mu_x b(d^+_x,\phi^{\fR}(t),\theta)  \cS^\fR_{x,\theta}(t)\bigr|\top 0.
\end{align}

Note that the total number of solvent agents at time $t$ satisfies
$$S^{(n)}(t)=\sum_{x\in\cX}\sum_{\theta=0}^{d^+_x}S^{(n)}_{x,\theta}(t),$$
which is dominated by $\sum_{x\in\cX}(d^+_x+d^-_x)\sum_{\theta=0}^{d^+_x}S^{(n)}_{x,\theta}(t)$. Then, by the convergence results \eqref{S_limit} and \eqref{S_sum_limit}, we obtain
$$\sup\limits_{t\leq \tau_n}\bigl|\frac{S^{(n)}(t)}{n}-f_S(t)\bigr|\top 0.$$
Further, from $D^{(n)}(t)=n-S^{(n)}(t)$, the number of ruined agents at time $t$ also satisfies
$$ \sup\limits_{t\leq \tau_n}\bigl|\frac{D^{(n)}(t)}{n}-f_D(t)\bigr|\top 0.$$

Finally, the total number of remaining ruinous outgoing half-edges at time $t$ is given by
$$W_n(t)=n\lambda^{(n)}-\cI^{(n)}(t)-\sum_{x\in\cX}\sum_{\theta=0}^{d^+_x}d^-_xS^{(n)}_{x,\theta}(t).$$ 
The same argument and \eqref{L_limit} imply that
\begin{align*}
\sup\limits_{t\leq \tau_n}\bigl|\frac{W_n(t)}{n}-f_W(e^{-t})\bigr|\top 0.
\end{align*}

This completes the proof of Theorem~\ref{thm:centrality}.

\subsection{Proof of Theorem~\ref{coro:centrality}}\label{proof:coro:centrality}
Similar to Example~\ref{eg:const_loss}, for any $x\in\cX$ and a realization of loss sequence $\bm{\ell}_x$ and shock $\epsilon_x$, we get
$$P^{\fR}_{x,\theta}(t,\epsilon_x,\bm{\ell}_x):=\PP(\tau_{x,0}(\epsilon_x, \bm{\ell}_x)=0 , U^{\fR,t}_{(1)}>\tau_{x,1}(\epsilon_x, \bm{\ell}_x),\ldots,U^{\fR,t}_{(\theta)}>\tau_{x,\theta}(\epsilon_x, \bm{\ell}_x))=\ind\{\theta < \delta_x(\epsilon_x, \bm{\ell}_x)\},$$
where $\delta_x(\epsilon_x, \bm{\ell}_x):=\inf\{\theta=0, \dots, d_x^+:\gamma_x(1-\epsilon_x)-\delta_x<\ell_{x,1}+\dots+\ell_{x,\theta}\}$ (by convention we set $\delta_x(\epsilon_x, \bm{\ell}_x):=\infty$ if there is not such a threshold $\theta$).
Hence, from the definition of the default threshold distribution $q_{x,\theta}$, $\cS^\fR_{x,\theta}(t)$ simplifies to
$$\cS^\fR_{x,\theta}(t)=\sum_{\delta=\theta+1}^{d^+_x}q_{x,\theta}+\bar{q}_x.$$
By applying Theorem~\ref{thm:centrality}, the first claim is established. Observing that
 $$b(d^+_x,\phi^\fR(t),\theta)=b(d^+_x,1-\phi^\fR(t),d^+_x-\theta),$$
we can rearrange the order of the sums to obtain
\begin{align*}
\sum_{\theta=1}^{d^+_x}b(d^+_x,\phi^\fR(t),\theta)\cS^\fR_{x,\theta}(t) & = \sum_{\theta=1}^{d^+_x}\bigl[\sum_{\delta=\theta+1}^{d^+_x}q_{x,\theta}+\bar{q}_x\bigr]b(d^+_x,1-\phi^\fR(t),d^+_x-\theta)\\
%& =\bar{q}+\sum_{\delta=1}^{d_x^+}\pi_x(\delta) \sum _{\theta=\delta-1}^{d^+_x}b\bigl(d_x^+, 1-\phi^\fR(t), d_x^+-\theta\bigr) \\
& =\bar{q}_x+\sum_{\theta=1}^{d_x^+}q_{x,\theta} \beta\bigl(d_x^+, 1-\phi^\fR(t), d_x^+-\theta+1\bigr).
\end{align*}
This leads to the conclusion for the limit function for  $\widehat{f}^\fR_S(t)$. Through similar arguments and calculation, the limit functions for the other cases can also be derived.

\subsection{Proof of Lemma~\ref{lem:tau}}\label{proof:lem:tau}
Recall that $$f^\fR_{W}(t):= \lambda(1-\phi^\fR(t)) - \sum_{x\in \cX} \mu_x d_x^-\sum_{\theta=0}^{d_x^+}b(d^+_x,\phi^{\fR}(t),\theta)\cS^\fR_{x,\theta}(t).$$
Consider a constant $t_1\in(0,  t_\fR^\star)$. Due to the continuity of $f^\fR_W(t)$ on the interval $[0,\infty)$, it follows that $f^\fR_W(t)>0$ on $[0,t_1)$. Thus, there exists a constant $C_1>0$ such that $f^\fR_W(t)\geq C_1$ for all $t\leq t_1$.

Since $W_n(\tau_n^{\star})=0$, if $\tau_n^{\star}\leq t_1$, then we have $W_n(\tau_n^{\star})/n-f^\fR_W(\tau^\star_n)\leq -C_1$. But on the other hand, according to Theorem~\ref{thm:centrality}, we have
$$\sup\limits_{t\leq \tau_n^{\star}}|\frac{W_n(t)}{n}-f^\fR_{W}(t)|\top 0.$$
This is a contradiction. Hence, it must be the case that
$\mathbb{P}(\tau_n^{\star}\leq t_1)\longrightarrow 0$, as $n\longrightarrow\infty$.
In the case where $t^\star_\fR=\infty$, we can choose any finite $t_1$ arbitrarily large, which implies that $\tau_n^{\star}\top\infty$.

We now consider the other scenario. Fix a constant $t_2\in (t^\star_\fR,t^\star_\fR+\varepsilon)$. Using a similar argument as above, we can show that there exists some constant $C_2>0$ such that $W_n(\tau_n^{\star})/n-f^\fR_W(\tau^\star_n)\geq C_2$ if $\tau_n^{\star}\geq t_2$. Therefore, $\mathbb{P}(\tau^\star_n \geq t_2)\longrightarrow 0$ as $n\longrightarrow\infty$. As $t_1$ and $t_2$ are arbitrary, letting both $t_1$ and $t_2$ tend to $t^\star_\fR$, we have $\tau^\star_n\top t^\star_\fR$.

\subsection{Proof of Theorem~\ref{thm:final}}\label{proof:thm:final}
By using Lemma~\ref{lem:tau}, we have that the stopping time $\tau^\star_n$ converges $t^\star_\fR$ in probability. In combination with the  continuity of $\cS^{\fR}_{x,\theta}(t)$ on $t$ and Theorem~\ref{thm:centrality},
\begin{align*}
\sup\limits_{t\leq \tau^\star_n}\bigl|\frac{D_{x}^{(n)}(t)}{n}-\mu_x \bigl(1-\sum_{\theta=0}^{d_x^+} b(d^+_x,\phi^{\fR}(t),\theta)  \cS^{\fR}_{x,\theta}(t) \bigr)\bigr|\top 0,
\end{align*}
as $n\to \infty$. Besides, If $\int_0^{t^\star_\fR} |\fR(s)| ds=\lambda$, that means at the stopping time $\tau^\star_n$, we have revealed almost all the outgoing half-edges. Thus the number of defaults must be $n-o_p(n)$.
Moreover, by Lemma~\ref{lem:tau}, we have $\tau^\star_n \top \infty$. Notice that $\cS^\fR_{x,\theta}(t)$ is non-decreasing and can not be larger than one. Thus there exists a limit for $\cS^\fR_{x,\theta}(t)$ when $t\to\infty$. Note also that, if $\|\fR\|_{L^1}<\lambda$, $\phi^{\fR}(t)\to \|\fR\|_{L^1}/\lambda$ as $t\to \infty$. Then the theorem follows by Lemma~\ref{lem:tau} and Theorem~\ref{thm:centrality}.

\subsection{Proof of Theorem~\ref{thm:main}}\label{sec:proof_main}
{\bf Proof of $(i)$.} The proof of $(i)$ will be divided into two parts.

{\bf Lipschitz continuity.} Let $\|\cdot\|_t$ denote the truncated $\mathbb{L}^1$ norm up to $t$, i.e., for a function $f$ defined on $\RR^+$, $\|f\|_t=\int_0^t |f(s)|ds$. The first step is to prove that $\Psi$ has some special Lipschitz continuity with respect to the first parameter. To be precise, for two different $\fR_1,\fR_2\in \mathbb{L}_\lambda(\RR^+)$ and all $t>0$, there exists a constant $C_0$, such that
\begin{equation}\label{Lip}
|\Psi(\fR_1)(t)-\Psi(\fR_2)(t)|\leq C_0\|\fR_1-\fR_2\|_t.
\end{equation}

For the first term, it is clear that
$$\lambda |\phi^{\fR_1}(t)-\phi^{\fR_2}(t)|\leq\int_0^t|\fR_1(s)-\fR_2(s)| ds.$$ 

We now analyze the difference $|\cS^{\fR_1}_{x,\theta}(t)-\cS^{\fR_2}_{x,\theta}(t)|$.  Let us recall the definitions of 
$$P^{\fR}_{x,\theta}(t,\epsilon_x,\bm{\ell}_x):=\PP(\tau_{x,0}(\epsilon_x, \bm{\ell}_x)=0 , U^{\fR,t}_{(1)}>\tau_{x,1}(\epsilon_x, \bm{\ell}_x),\ldots,U^{\fR,t}_{(\theta)}>\tau_{x,\theta}(\epsilon_x, \bm{\ell}_x)),$$
and
$$\cS^{\fR}_{x,\theta}(t)=\PP(\tau_{x,0}(\epsilon_x,\bm{L}_x)=0, U^{\fR,t}_{(1)}>\tau_{x,1}(\epsilon_x, \bm{L}_x),\ldots,U^{\fR,t}_{(\theta)}>\tau_{x,\theta}(\epsilon_x, \bm{L}_x)).$$

For each realization of loss sequence $\bm{\ell}_x$ and shock $\epsilon_x$, let $\cA_{x,\theta}(\epsilon_x,\bm{\ell}_x)$ be the set of all $(u_1,u_2,\ldots,u_{\theta})$ such that the rearranged sequence $(u_{(1)},u_{(2)},\ldots,u_{(\theta)})$ in increasing order satisfies 
$$u_{(1)}>\tau_{x,1}(\epsilon_x, \bm{\ell}_x), \ldots,u_{(\theta)}>\tau_{x,\theta}(\epsilon_x, \bm{\ell}_x).$$
Let $\fR^t(s)=\fR(s)/\|\fR\|_t$. We have
\begin{align*}
& \int_{\cA_{x,\theta}(\epsilon_x,\bm{\ell}_x)}\fR^t_1(s_1)\cdots \fR^t_1(s_\theta)ds_1 \cdots ds_\theta-\int_{\cA_{x,\theta}(\bm{\ell}_x)}\fR^t_2(s_1)\cdots \fR^t_2(s_\theta)ds_1\cdots ds_\theta \\
\leq & \int_{\cA_{x,\theta}(\bm{\ell}_x)}| \fR^t_1(s_1)\cdots \fR^t_1(s_\theta)-\fR^t_2(s_1)\cdots \fR^t_2(s_\theta)|ds_1 \cdots ds_\theta \\
\leq &  \int_{[0,t]^\theta}| \fR^t_1(s_1)\cdots \fR^t_1(s_\theta)-\fR^t_2(s_1)\cdots \fR^t_2(s_\theta)|ds_1 \cdots ds_\theta.
\end{align*}
 By adding and subtracting terms, it is easy to show that
 
% For writing convenience, we take $\theta=3$.
%\begin{align*}
%& | \fR^t_1(s_1)\fR^t_1(s_2)\fR^t_1(s_3)-\fR^t_2(s_1)\fR^t_2(s_2)\fR^t_2(s_3)| \\
%\leq & | \fR^t_1(s_1)\fR^t_1(s_2)\fR^t_1(s_3)- \fR^t_2(s_1)\fR^t_1(s_2)\fR^t_1(s_3)|+| \fR^t_2(s_1)\fR^t_1(s_2)\fR^t_1(s_3)- \fR^t_2(s_1)\fR^t_2(s_2)\fR^t_1(s_3)| \\
%& \quad +| \fR^t_2(s_1)\fR^t_2(s_2)\fR^t_1(s_3)- \fR^t_2(s_1)\fR^t_2(s_2)\fR^t_2(s_3)| \\
%\leq & | \fR^t_1(s_1)- \fR^t_2(s_1)|\fR^t_1(s_2)\fR^t_1(s_3)+|\fR^t_1(s_2)- \fR^t_2(s_2)|\fR^t_2(s_1)\fR^t_1(s_3) \\
%& \quad+|\fR^t_1(s_3)- \fR^t_2(s_3)| \fR^t_2(s_1)\fR^t_2(s_2).
%\end{align*}
%Then it is easily seen that 
%$$\int_{[0,t]^3}| \fR^t_1(s_1)- \fR^t_2(s_1)|\fR^t_1(s_2)\fR^t_1(s_3)ds_1 ds_2 ds_3=\|\fR^t_1-\fR^t_2\|_t,$$
%and the same for the other two terms. Thus apply the same method to any $\theta$, we obtain
$$ \int_{[0,t]^\theta}| \fR^t_1(s_1)\cdots \fR^t_1(s_\theta)-\fR^t_2(s_1)\cdots \fR^t_2(s_\theta)|ds_1 \cdots ds_\theta \leq \theta \|\fR^t_1-\fR^t_2\|_t.$$
Notice that there exists a constant $C_1$ large enough such that
\begin{align*}
\|\fR^t_1-\fR^t_2\|_t \leq & \|\frac{\fR_1}{\|\fR_1\|_t}-\frac{\fR_2}{\|\fR_1\|_t}\|_t+\|\frac{\fR_2}{\|\fR_1\|_t}-\frac{\fR_2}{\|\fR_2\|_t}\|_t \\
\leq & \frac{\|\fR_1-\fR_2\|_t}{\|\fR_1\|_t}+\frac{\|\fR_2\|_t-\|\fR_1\|_t}{\|\fR_1\|_t\|\fR_2\|_t}\|\fR_2\|_t\\
\leq  & C_1 \|\fR_1-\fR_2\|_t.
\end{align*}
Therefore, for any realization of loss sequence $\bm{\ell}_x$ and shock $\epsilon_x$, we have 
$$\PP(\tau_{x,0}(\epsilon_x, \bm{\ell}_x)=0 , U^{\fR,t}_{(1)}>\tau_{x,1}(\epsilon_x, \bm{\ell}_x),\ldots,U^{\fR,t}_{(\theta)}>\tau_{x,\theta}(\epsilon_x, \bm{\ell}_x))\leq C_1\theta \|\fR_1-\fR_2\|_t.$$
It therefore follows that
\begin{equation}\label{eq:P}
\PP(\tau_{x,0}(\epsilon_x,\bm{L}_x)=0, U^{\fR,t}_{(1)}>\tau_{x,1}(\epsilon_x, \bm{L}_x),\ldots,U^{\fR,t}_{(\theta)}>\tau_{x,\theta}(\epsilon_x, \bm{L}_x))\leq C_1\theta \|\fR_1-\fR_2\|_t.
\end{equation}

On the other hand, by elementary calculation for the derivative of $b(d^+_x, y ,\theta)$ with respect to $y$, 
$$\frac{\partial b(d^+_x,y,\theta)}{\partial y}=d^+_xb(d^+_x-1,y,\theta-1)-d^+_xb(d^+_x-1,y,\theta),$$
and 
$$\frac{\partial b(d^+_x,y,0)}{\partial y}=-d^+_x b(d^+_x-1,y,0),\quad\quad\frac{\partial b(d^+_x,y,d^+_x)}{\partial y}=d^+_xb(d^+_x-1,y,d^+_x-1).$$
For each $y\in[0,1]$, there exists a threshold $\bar{\theta}_y$ such that $\partial b(d^+_x,y,\theta)/\partial y\leq 0$ for $\theta\leq \bar{\theta}_y$ and $\partial b(d^+_x,y,\theta)/\partial y \geq 0$ for $\theta > \bar{\theta}_y$. Therefore we have for all $x\in\cX$,
\begin{align}\label{eq:deriv}
\frac{\partial \sum_{\theta=0}^{d^+_x} a_{x,\theta} b(d^+_x, y,\theta)}{\partial y} \leq 2 d^+_x b(d^+_x-1,y, \bar{\theta}_y)\leq 2 d^+_x,
\end{align}
if $a_{x,\theta}$ is a constant coefficient satisfying $0\leq a_{x,\theta} \leq1$ for all $(x,\theta)$. 
 
Then by adding and subtracting terms and combining \eqref{eq:P} and \eqref{eq:deriv}, we conclude that there exists a constant $C_2$ such that for each $x\in\cX$,
\begin{align*}
\bigl|\sum_{\theta=0}^{d^+_x}(b(d^+_x,\phi^{\fR_1}(t),\theta)&\cS^{\fR_1}_{x,\theta}(t)- b(d^+_x,\phi^{\fR_2}(t),\theta)\cS^{\fR_2}_{x,\theta}(t))\bigr| \\
%&\leq \sum_{\theta=0}^{d^+_x} b(d^+_x,\phi^{\fR_2}(t),\theta) \bigl| \PP_x(U^{\fR_1,t}_{(1)}>\tau_{x,1}(\bm{L}_x),U^{\fR_1,t}_{(2)}>\tau_{x,2}(\bm{L}_x),\ldots,U^{\fR_1,t}_{(\theta)}>\tau_{x,\theta}(\bm{L}_x)) \\
 %& \quad \quad -\PP_x(U^{\fR_2,t}_{(1)}>\tau_{x,1}(\bm{L}_x),U^{\fR_2,t}_{(2)}>\tau_{x,2}(\bm{L}_x),\ldots,U^{\fR_2,t}_{(\theta)}>\tau_{x,\theta}(\bm{L}_x)) \bigr| \\
 %& \quad  +\bigl|\sum_{\theta=0}^{d^+_x}\bigl(b(d^+_x,\phi^{\fR_1}(t),\theta)-b(d^+_x,\phi^{\fR_2}(t),\theta) \bigr|  \PP_x(U^{\fR_1,t}_{(1)}>\tau_{x,1}(\bm{L}_x),U^{\fR_1,t}_{(2)} \\ 
 %& \quad \quad >\tau_{x,2}(\bm{L}_x),\ldots,U^{\fR_1,t}_{(\theta)}>\tau_{x,\theta}(\bm{L}_x))   \bigr|\\
 & \leq \sum_{\theta=0}^{d^+_x} b(d^+_x,\phi^{\fR_2}(t),\theta) (C_1 \theta) \|\fR_1-\fR_2\|_t + 2 d^+_x |\phi^{\fR_1}(t)-\phi^{\fR_2}(t)| \\
& \leq   \sum_{\theta=0}^{d^+_x} b(d^+_x,\phi^{\fR_2}(t),\theta) C_1 d^+_x \|\fR_1-\fR_2\|_t + 2\frac{ d^+_x}{\lambda} \|\fR_1-\fR_2\|_t  
\leq C_2 d^+_x \|\fR_1-\fR_2\|_t.
\end{align*}

%\begin{align*}
%\sum_{\theta=0}^{d^+_x}|P^{\fR_1}_{x,\theta}(t,\bm{\ell}^{\theta}_x)-P^{\fR_2}_{x,\theta}(t,\bm{\ell}^{\theta}_x)|\leq d^+_x \|\fR_1-\fR_2\|_t.
%\end{align*}

It therefore follows from Assumption~\ref{cond-moment} that there exists a constant $C_0$ such that
$$\sum_{x\in \cX} \mu_x q_x d_x^- \bigl|\sum_{\theta=0}^{d_x^+}(\cS^{\fR_1}_{x,\theta}(t)-\cS^{\fR_2}_{x,\theta}(t))\bigr|\leq C_1 \sum_{x\in \cX} \mu_x q_x d_x^- d_x^+\|\fR_1-\fR_2\|_t\leq C_0\|\fR_1-\fR_2\|_t,$$
and hence \eqref{Lip} holds.

{\bf Existence and uniqueness.} We show existence using a standard iterative procedure. To prove the existence and uniqueness on the non-negative real numbers $\mathbb{R}^+$, it is sufficient to prove them on the interval $[0,T]$ for any arbitrary $T>0$. Specifically, let $h^{(0)}(t)=\beta\sum_{x\in \cX}\mu_x d^-_x(1-q_{x,0})-t$, and define
\begin{equation}\label{define_ite}
h^{(n)}(t)=\beta\Psi(h^{(n-1)})(t).
\end{equation}
By \eqref{Lip}, we have (for $v\in \RR_+$)
\begin{equation*}
|h^{(n)}(v)-h^{(n-1)}(v)| \leq  \beta C_0 \|h^{(n-1)}-h^{(n-2)}\|_v.
\end{equation*}
Integrating the above formula from $v=0$ to $v=t$ gives us
\begin{equation}\label{ite}
\begin{split}
\|h^{(n)}-h^{(n-1)}\|_t \leq \beta C_0\int_0^t\|h^{(n-1)}-h^{(n-2)}\|_vdv.
\end{split}
\end{equation}
Clearly $\|h^{(1)}-h^{(0)}\|_T\leq C$ for some constant $C$. Then by iterating the above formula \eqref{ite}, we have 
$$\|h^{(n)}-h^{(n-1)}\|_t \leq \frac{C}{(n-1)!}(\beta C_0)^{n-1}t^{n-1}.$$
Moreover, for some constant $C(\beta,K,T)$ depending on $\beta$, $K$ and $T$, the infinite sum satisfies
$$\sum_{n=1}^{\infty}\|h^{(n)}-h^{(n-1)}\|_T\leq C(\beta,C_0,T)e^{C(\beta,C_0,T)}.$$
Therefore the series $h^{(n)}(t)$ converges in the $L^1([0,T])$ space to a limit $\fF(t)$. Due to the continuity of the limit, we can conclude the existence of a solution.

To prove uniqueness, suppose there exists two different solutions $h_1(t)$ and $h_2(t)$ satisfying the fixed point equation. We have
$$\|h_1-h_2\|_T \leq C_0\int_0^T \|h_1-h_2\|_vdv.$$  
Since the function $\|h_1-h_2\|_t$ is bounded on $[0,T]$ and positive, the Gronwall Lemma implies that $\|h_1-h_2\|_t=0$ on $[0,T]$. Since $T$ is arbitrary and the solution is continuous, uniqueness follows. 

{ \bf Proof of $(ii)$.} 
By the construction of our networked risk processes, the loss reveal function $\cR_n(t)$ is equal to the $\beta W_n(t)$. Suppose that the limit process of $W_n(t)$ exists and satisfies the conditions outlined in Theorem~\ref{thm:centrality}. From the definition of $f^{\fR^\star}_W(t)$, we have
$$\frac{\fR^\star(t)}{\beta}=f^{\fR^\star}_W(t)=\lambda \phi^{\fR^\star}(t) - \sum_{x\in \cX} \mu_x d_x^-\sum_{\theta=0}^{d_x^+}b(d^+_x,\phi^{\fR^\star}(t),\theta)  \cS^{\fR^\star}_{x,\theta}(t).$$
The existence and uniqueness of the solution to this fixed point equation were proved in point $(i)$. As a result, by Theorem~\ref{thm:centrality}, the unique solution $\fR^\star$ of the fixed point equation in $(i)$  is the limit process of $\beta W_n(t)$ and we have 
 \begin{align*}
\sup\limits_{t\leq \tau^\star_n}\bigl|\frac{W_n(t)}{n}-\frac{\fR^\star(t)}{\beta}\bigr|\top 0.
\end{align*}
The remaining limit results follow directly from Theorem~\ref{thm:centrality}.

The proof of Theorem~\ref{thm:final} is now complete.  

\subsection{Proof of Lemma~\ref{lem:tau_W}}
%It is easily seen that for a given fraction of initially ruined agents $\bm{\epsilon}_n>\bm{0}$, the number of initial defaults is of order $n$. 
Let $\alpha_n n$ denote the number of initially ruined outgoing half-edges. Note that, since $\fR^\star(0)=\beta \sum_{x\in\cX}\mu_x d^-_x(1-q_{x,0})>0$, we have $\lim\inf_n \alpha_n>0$. Let $\Lambda_n$ denote the time required to reveal all the $\alpha_n n$ initially ruined outgoing half-edges without incurring any new ruined agents, i.e. 
$$\Lambda_n=T^{(n)}_1+T^{(n)}_2+\cdots+ T^{(n)}_{\alpha_n n},$$
where $T^{(n)}_k$ is the time duration of the $k$-th reveal and is an exponential random variable with parameter $\alpha_n n-k+1$, and they are independent. If no new ruins incurred at each step, $\tau^\star_n$ will attain the smallest possible value stochastically, namely $\tau^\star_n\geq_{\mbox{st}} \Lambda_n$. It is therefore sufficient to prove that $\Lambda_n>(1-\epsilon)\log n$ with high probability for any $\epsilon >0$.
By Markov's inequality we have
$$\PP(\Lambda_n\leq (1-\epsilon)\log n)=\PP(e^{-\Lambda_n}\geq e^{-(1-\epsilon)\log n})\leq e^{(1-\epsilon)\log n}\prod_{k=1}^{\alpha_n n} \EE e^{-T^{(n)}_k}.$$
Since $\EE e^{-T^{(n)}_k}=(\alpha_n n-k+1)/(\alpha_n n-k)$, we have
\begin{align*}
\PP(\Lambda_n\leq (1-\epsilon)\log n)\leq & \exp\{(1-\epsilon)\log n-\sum_{k=1}^{\alpha_n n}\log (1+\frac{1}{k}) \} .
\end{align*} 
By applying Taylor expansion to $\log(1+1/k)$, it follows that
%\begin{align*}
%\sum_{k=1}^{a_n n}\log (1+\frac{1}{k})= &\sum_{k=1}^{a_n n} (\frac{1}{k}+o(\frac{1}{k}))=\log (a_n n)+o(\log n) .
%\end{align*}  
%It follows that
\begin{align*}
\PP(\Lambda_n\leq (1-\epsilon)\log n)\leq & \exp\{(1-\epsilon)\log n- \log (\alpha_n n)-o(\log n) \} \\
\leq & \exp\{(1-\epsilon)\log n- \log n- \log \alpha_n-o(\log n)= O (n ^{-\epsilon}).
\end{align*} 
We thus have $\Lambda_n>(1-\epsilon)\log n$ with high probability for any $\epsilon >0$. The proof is complete.

\section{Complex Networked Risk Processes}
\label{sec:extensions}

So far we have assumed that the external debt is a constant function for each agent. It will be interesting to extend the model by considering a dynamics for this external debt that is itself like in the classical Cram\'er-Lundberg model. Namely, we assume that the external debt for agent $i\in[n]$ follows  $\delta_i(t)= \sum_{j=1}^{N_i(t)} \zeta_{i}^{(j)}$, where $N_i(t)$ is a Poisson process with intensity $\beta_i$ and the claim sizes $\{\zeta_{i}^{(j)}\}_{j=1}^\infty$ are i.i.d. distributed random variables with distribution $G_i$ with finite positive mean and variance.

The risk process for the capital of agent $i\in [n]$ with network interactions $\cG_n$ follows
\begin{align}\label{eq:riskP-comp}
C_i(t):=\gamma_i(1-\epsilon_i)+\alpha_i (t)-\sum_{j=1}^{N_i(t)} \zeta_{i}^{(j)}-\sum_{j\in [n]: j\to i} L_{ji}\ind\{\tau_j+T_{ji} \leq t\},
\end{align}
 where, similar to \eqref{eq:riskP} we  assume that $T_{ji}\sim \mathrm{Exp}(\beta)$ are i.i.d. exponentially distributed with some parameter $\beta>0$ for all $i,j \in [n]$. 

Consider the node classification of Section~\ref{sec:class} and assume that for all agents of the same type $x\in\cX$, the associated external risk process has the same features. Therefore,  agents of type $x$ have the same claim distribution, denoted  by $G_x$ and external claim arrival intensity denoted by $\beta_x$.
%with mean $\widetilde{\mu}_x>0$ and finite variance $\widetilde{\sigma}_x^2$.  

Let us define the external risk process for an agent $i$  of type $x\in \cX$ with initial capital $u$ by
$$C_x^{\rm EX} (u,t):=u+\alpha_x (t)-\sum_{j=1}^{N_x(t)} \zeta_{i}^{(j)},$$
where $N_x(t)$ is a Poisson process with intensity $\beta_x$ and the claim sizes $\{\zeta_{i}^{(j)}\}_{j=1}^\infty$ are i.i.d. distributed random variables with distribution $G_x$ and mean $\bar{\zeta}_x>0$. Similarly, for the network $\cG_n$ and initial type-dependent capital vector ${\bf u}=(u_x, x\in \cX)$, we define the internal risk process for agent $i$ by
$$C_x^{\rm IN}({\bf u}, t):=u_x+\alpha_x (t)-\sum_{j\in [n]: j\to i} L_{ji}\ind\{\tau_j+T_{ji} \leq t\},$$
where losses $L_{ij}$ are i.i.d. random variables with distribution $F_x$ and  $T_{ji}\sim \mathrm{Exp}(\beta)$ are i.i.d. exponentially distributed. 

For simplicity we assume that $\epsilon_x$ is a constant and $\alpha_x(t)=\alpha_x t$ in this section.

Let us denote by $$\psi_x^{\rm EX}(u,t)=\PP(C_x^{\rm EX}(u, s)\leq 0, \ \text{for some} \ s\leq t)$$
and 
$$\psi_x^{\rm IN}({\bf u}, t)=\PP(C_x^{\rm IN}({\bf u}, s)\leq 0, \ \text{for some} \ s\leq t),$$
the ruin probabilities for the external and respectively internal risk processes of an agent of type $x$.
The ruin probability for the external process is a well studied problem whose solution we review below.
The ruin probability for the internal process is given by Theorem \ref{thm:main} in the limit when $n$ is large and converges to 
$$\psi_x^{\rm IN}({\bf u}, t)=1-\sum_{\theta=0}^{d^+_x}b(d^+_x,\phi^{\fR^\star}(t),\theta) \cS^{\fR^\star}_{x,\theta}(\infty)+o_p(1).$$
Note that the parameters (and in particular the fraction of 
initial ruined agents) depend on the initial capital levels ${\bf u}$, which reflect any initial shock and external debt of our baseline process in Section \ref{sec:ruin}. We will use these ruin probabilities to provide an upper bound and lower bound on the ruin probability for agent $i$ of type $x$ 
given by 
$$\psi_x(t):=\PP(C_x(s)\leq 0, \ \text{for some} \ s\leq t),$$
where
\[C_x(t):=\gamma_x(1-\epsilon_x)+\alpha_x t-\sum_{j=1}^{N_x(t)} \zeta_{i}^{(j)}-\sum_{j\in [n]: j\to i} L_{ji}\ind\{\tau_j+T_{ji} \leq t\}.\]

 It is known (see e.g.~\cite{asmussen2010ruin, embrechts2013modelling}) that whenever $\beta_x \bar{\zeta}_x>\alpha_x$, we have $\psi^{\rm EX}_x(u, \infty)=1$ for all $u\in \RR$ and whenever $\beta_x \bar{\zeta}_x<\alpha_x$, the final ruin probability can be computed using the famous Pollaczek–Khinchine formula as
\begin{align}
\psi_x^{\rm EX}(u,\infty)=\left(1-\frac{\beta_x \bar{\zeta}_x}{\alpha_x}\right)\sum_{k=0}^\infty \left(\frac{\beta_x \bar{\zeta}_x}{\alpha_x}\right)^k\left(1-\widehat{G}_x^{*k}(\gamma)\right),
\label{eq:classickramer}
\end{align}
where 
\[\widehat{G}_x(u)=\frac{1}{\bar \zeta_x} \int_0^u \bigl(1-G_x(u)\bigr) du,\]
and the operator $(\cdot)^{*k}$ denotes the $k$-fold convolution.

 Furthermore, for the finite time horizon ruin probability  $\psi^{\rm EX}_x(u,t)$, we have the Seal-type formula, see e.g. \cite[Proposition 3.4]{lefevreloisel}, 
for any $t>0$ and any $x\in\cX$, 
\begin{align*}
\psi^{\rm EX}_x(\beta_x u,t)= & 1-e^{-\beta_x t}-\int_0^{u+t}f_t(z)dz+\int_u^{u+t}e^{-\beta_x(t+u-z)}f_{z-u}(z)dz \\
& \quad +\int_u^{u+t}f_{z-u}(z)\bigl( \int_z^{u+t} \frac{t+u-y}{t+u-z} f_{t+u-z}(y-z)dy\bigr)dz,  
\end{align*}
where $f^x_t(\cdot)$ denotes the density of $\sum_{j=1}^{N_x(t)} \zeta_{i}^{(j)}/\beta_x$ on $(0,\infty)$.

The following theorem provides upper and lower bounds for the complex ruin probability $\psi_x(t)$ by utilizing $\psi_x^{\rm IN}$ and $\psi_x^{\rm EX}$. We define $\widetilde{\tau}_{x,\theta}$ similarly to $\tau_{x,\theta}$ in Equation \eqref{eq:deftau}, but with $\alpha_x$ and $\gamma_x$ replaced by $\alpha_x/2$ and $\gamma_x/2$, respectively. Furthermore, for any $\fR\in\mathbb{L}_{\lambda}(\mathbb{R}^+)$, we define $\widetilde{\cS}^\fR_{x,\theta}(t)$ in the same manner as $\cS^\fR_{x,\theta}(t)$, but with $\tau_{x,\theta}$ replaced by $\widetilde{\tau}_{x,\theta}$.

\begin{theorem}\label{thm:ruin}
Let $\fR^\star$ be the unique solution of the fixed point equation in Theorem~\ref{thm:main} with an initial value $\fR^\star(0)=\beta \sum_{x\in\cX}\mu_x d^-_x(1-q_{x,0})$, and let $\widetilde{\fR}^\star$ be the unique solution of
$$\widetilde{\fR}^\star(t)= \lambda(1- \phi^{\widetilde{\fR}^\star}(t)) - \sum_{x\in \cX} \mu_x d_x^-\sum_{\theta=0}^{d_x^+}b(d^+_x,\phi^{\widetilde{\fR}^\star}(t),\theta)\widetilde{\cS}^{\widetilde{\fR}^\star}_{x,\theta}(t),$$
with an initial value $\widetilde{\fR}^\star(0)=\sum_{x\in\cX}\mu_x d^-_x(1-q_{x,0})$. Also, let $\psi_x^{\rm EX}(t):=\psi_x^{\rm EX}(\gamma_x(1-\epsilon_x), t)$ and let $\widetilde{\psi}_x^{\rm EX}(t)$ be the ruin probability for the external risk process at time $t$ staring with half the initial value $\gamma_x(1-\epsilon_x)/2$ and half the capital growth rate $\alpha_x/2$. Under Assumption~\ref{cond-limit} and Assumption~\ref{cond-moment}, for all $x\in\cX$ and for any $t>0$, when $n$ is large enough, we have 
$$\psi_x(t)\geq 1-(1-\psi^{\rm EX}_x(t))\sum_{\theta=0}^{d^+_x}b(d^+_x,\phi^{\fR^\star}(t),\theta)\cS^{\fR^\star}_{x,\theta}(t),$$
%$$\Phi_x(t)\geq q_x\Phi^{EX}_x(t)\sum_{\theta=0}^{d^+_x}\cS^{\fF^1}_{x,\theta}(t)+ (1-q_x\sum_{\theta=0}^{d^+_x}\cS^{\fF^1}_{x,\theta}(t)),$$
and
%$$\Phi_x(t)\leq \mu_x(1-q_x(1-\Phi^{EX}_x(t))\sum_{\theta=0}^{d^+_x}\cS^{\fF^2}_{x,\theta}(t)),$$
$$\psi_x(t)\leq 1-(1-\widetilde{\psi}^{\rm IN}_x(t))\sum_{\theta=0}^{d^+_x}b(d^+_x,\phi^{\widetilde{\fR}^\star}(t),\theta)\widetilde{\cS}^{\widetilde{\fR}^\star}_{x,\theta}(t).$$
%where $\fF^1:=\fF(\sum_{x\in\cX}d^-_x(1-q_x))$ and $\fF^2:=\widetilde{\fF}(\sum_{x\in\cX}d^-_x(1-q_x(1-\widetilde{\Phi}^1_x(t))))$.
\end{theorem}

\begin{proof}
We write all the quantities in their limit forms, since this theorem states the bounds when $n$ is large enough. We define the events $$A_t:=\{C^{\rm EX}_x(s)> 0, \mbox{for all }s\leq t\},$$ $$B_t:=\{C^{\rm IN}_x(s)>0, \mbox{for all }s\leq t \mbox{ with an initial fraction of ruinous half-edges } \sum_{x\in\cX}\mu_x d^-_x(1-q_{x,0})\}$$ and, $$D_t:=\{C_x(s)> 0, \mbox{for all }s\leq t \mbox{ with an initial fraction of ruinous half-edges } \sum_{x\in\cX}\mu_x d^-_x(1-q_{x,0})\}.$$ 
Then by the independence of external and internal risk processes, we have
$$\PP(D_t)\leq \PP(A_t\cap B_t)=\PP(A_t) \cdot\PP(B_t)=(1-\psi^{\rm EX}_x(t))\sum_{\theta=0}^{d^+_x}b(d^+_x,\phi^{\fR^\star}(t),\theta)\cS^{\fR^\star}_{x,\theta}(t).$$

Since $\psi_x(t)= \PP(D^{c}_t)$, then we obtain the lower bound in the theorem. To prove the upper bound, we define the following events
\begin{align*}
\widetilde{A}_t:=&\{\sum_{j=1}^{N_i(t)} \zeta_{i}^{(j)}<\frac{\gamma_x}{2}+\frac{\alpha_x(s)}{2}, \mbox{for all }s\leq t\},\\
\widetilde{B}_{T,t}:=&\{ L_{ji}\ind\{\tau_j+T_{ji} \leq t\}<\frac{\gamma_x}{2}+\frac{\alpha_x(s)}{2} , \mbox{for all }s\leq t \mbox{ with the loss reveal intenstiy }\widetilde{\fR}^\star \}.
\end{align*}
%and
%$$\widetilde{D}_{T,t}:=\{\widetilde{C}_x(s) >0 , \mbox{for all }s\leq t \mbox{ with loss reveal function }\widetilde{\fR}^\star_T\}.$$
Then we have
$$\PP(D_{t})\geq \PP(\widetilde{A}_t\cap \widetilde{B}_{T,t})=\PP(\widetilde{A}_{t}) \cdot\PP(\widetilde{B}_{T,t})=(1-\widetilde{\psi}^{\rm EX}_x(t))\sum_{\theta=0}^{d^+_x}b(d^+_x,\phi^{\fR^\star}(t),\theta) \widetilde{\cS}^{\widetilde{\fR}^\star_T}_{x,\theta}(t).$$
Since $\psi_x(t) = \PP(D^{c}_t)$, the upper bound in theorem follows.

\end{proof}

The general setup of Section~\ref{sec:general} for the internal risk process can be considered where the limit function for the loss reveal intensity function is a given function $\fR$ that satisfies Assumption~\ref{cond-R}. In this case, both the upper and lower bounds in the theorem above still hold, with $\fR^\star$ and $\widetilde{\fR}^\star$ being replaced by $\fR$. The question of finding exact (asymptotic) ruin probabilities in this complex networked risk processes is left for future research.

%\begin{corollary}
%Under the same assumptions as inTheorem~\ref{thm:ruin}, suppose the loss reveal intensity function satisfying $\cR_n(t)=nc+o(n)$ for some constant $c$. Then we have as $n\to \infty$, for $t< \sum_{x\in\cX}d^-_x(1-q_x)/c$,
%$$\Phi_x(t)\to \mu_x(1-q_x(1-\Phi^{EX}_x(t))\sum_{\theta=0}^{d^+_x}\cS^{h(c)}_{x,\theta}(t)),$$ 
%where $h(c)\equiv c$ is a constant function.
%\end{corollary}

%The above theorem provides bounds for the ruin probabilities in complex open networked risk processes. 

\section{Concluding Remarks}\label{sec:conc}
We have extended and solved an open problem posed in our prior work \cite{amini2022dynamic}, namely a multi-dimensional extension to the Cram\'er-Lundberg risk processes with network-driven losses. We called this the (internal) networked risk processes. In this general case, the losses from the ruined neighbours can have any given distribution. There is a general arrival process which drives the losses stemming from the ruined agents.
This model opens the way to integrating two streams of literature, one being on risk models that have outside exogenous losses and one where losses are internal and interconnected. 

This is one step forward to study dynamic financial network models and at the same time bringing tractability to multi-dimensional Cram\'er-Lundberg risk processes.
We leave to the future the study of  further generalizations. Here we mention only a few of them.  In Section~\ref{sec:extensions} we introduced a model that allows for exogenous individual external risk processes in addition to the internal networked one. We have provided lower and upper bounds for the ruin probabilities. Finding the exact ruin probabilities in this case is left for future research. 
Instead of a sparse network structure that underlies the internal loss propagation model, we could consider risk processes in a dense network. In particular, since our results are asymptotic in nature, it seems promising to consider graphons. These  have been developed by Lov\'asz et al., see e.g.~\cite{lovasz2012, borgs2008convergent, borgs2012convergent},  as a natural continuum limit object for large dense graphs.
It would also be interesting to optimize the underlying graphon structure given the profile of external losses. When these external losses follow a general risk process, we can expect to have various types of adaptive graphons that minimize the loss contagion.

The model we consider  is flexible enough to incorporate multiple channels of contagion. For example firesales have been incorporated when the process does not have growth \cite{amicaosu2021fire}.
Incorporating firesales or other indirect contagion mechanisms adds another fixed point to the analysis, driven for example by the price dynamics of an illiquid asset.
In the context with growth, the entire path of the endogenous price process would have to be consistent  with the time threshold functions~\eqref{eq:deftau}, which in turn are price dependent.

Finally, key research questions revolve around the optimal dividend distributions. 
In the Cram\'er–Lundberg setting, this problem  has been studied extensively in the literature, see e.g., \cite{gerber1969entscheidungskriterien,azcue2005optimal,schmidli2006optimisation}.
For spectrally negative L\'evy processes, the optimal dividend distribution have been shown to be of constant barrier type, see e.g., \cite{aap,loeffen2008optimality}. The study of optimal dividend distributions in the presence of network risk remains an open problem.
One may replace ruin time by a  ruin time  observed with Poissonian frequency,  see for example \cite{albrecher2016exit,albrecher2017strikingly}, which is equivalent to a  Parisian ruin time with exponential grace period. Instead of dividend distribution towards outside entities, an alternative setup would be that of bail-ins. In such settings, agents would divert part of their growth to other agents with the goal of preventing their ruin. The optimal rate of contribution towards bailing in other nodes in the networks has been treated in the simple case of one firm and one subsidiary in \cite{avram2016management}. In the multi-dimensional setup one can consider a central node that covers the shortfall of the other nodes with the pool of capital to which all agents contribute, these are "bail-ins" (a term coined in 2010 as the opposite of bailouts).
The bail-in amounts are the necessary funds to ``reflect" the risk processes  back to an optimal positive level.
Consequently, the central node risk process is a marked point process in which the inter-arrival times  are the first passage times of the members capital process, and the jump sizes are the bail-in amounts.  This process has thus a non-trivial dependence structure between interarrival times and jump sizes, that we leave for future research.

\bibliographystyle{plain}
\bibliography{bib_networkruin}

\end{document}